\documentclass[12pt, reqno, a4paper]{amsart}

\usepackage[all]{xy}

\usepackage{ amssymb, amsmath, enumerate, amsfonts, amsthm, mathrsfs, url, bm, mathtools}

\numberwithin{equation}{section}

\usepackage{xcolor}  	
\usepackage[backref]{hyperref}
\hypersetup{
	colorlinks,
    linkcolor={blue!60!black},
    citecolor={blue!60!black},
    urlcolor={red!60!black}
}

\usepackage{color}

\usepackage[margin=1.25in]{geometry}

\RequirePackage{doi}

\usepackage[square,sort,comma,numbers]{natbib}
\makeatletter
\@namedef{subjclassname@2020}{%
  \textup{2020} Mathematics Subject Classification}
\makeatother

 \newcommand{\set}[1]{\left\{#1\right\}}

\newcommand{\bigabs}[1]{\bigl| #1 \bigr|}
\newcommand{\Bigabs}[1]{\Bigl| #1 \Bigr|}
\newcommand{\biggabs}[1]{\biggl| #1 \biggr|}
\newcommand{\Biggabs}[1]{\Biggl| #1 \Biggr|}

\newcommand{\floor}[1]{\left\lfloor #1 \right\rfloor}

\newcommand{\bigbrac}[1]{\bigl( #1 \bigr)}
\newcommand{\Bigbrac}[1]{\Bigl( #1 \Bigr)}
\newcommand{\biggbrac}[1]{\biggl( #1 \biggr)}

\newcommand{\norm}[1]{\left\| #1\right\|}

\newcommand{\pq}{\omega_{(Q_0,Q]}}
\newcommand{\rd}{\,\mathrm{d}}
\newcommand{\po}{\mathcal P_1}
\newcommand{\pt}{\mathcal P_2}

\newcommand{\twosum}[2]{ \sum_{\substack{#1\\ #2}}}

\newcommand{\N}{\mathbb{N}}

\newcommand{\R}{\mathbb{R}}
\newcommand{\C}{\mathbb{C}}
\newcommand{\T}{\mathbb{T}}

\newcommand{\eps}{\varepsilon}

\makeatletter
\let\@@pmod\pmod
\DeclareRobustCommand{\pmod}{\@ifstar\@pmods\@@pmod}
\def\@pmods#1{\mkern4mu({\operator@font mod}\mkern 6mu#1)}
\makeatother

\renewcommand{\mod}[1]{\,(\mathrm{mod}\,#1)}

\renewcommand{\leq}{\leqslant}
\renewcommand{\geq}{\geqslant}
\renewcommand{\epsilon}{\varepsilon}

\newtheorem{theorem}{Theorem}[section]

\newtheorem{corollary}[theorem]{Corollary}

\newtheorem{proposition}[theorem]{Proposition}
\newtheorem{lemma}[theorem]{Lemma}

\theoremstyle{definition}
\newtheorem{definition}[theorem]{Definition}
\newtheorem*{remark}{Remark}

\numberwithin{theorem}{section}

\begin{document}

\title[Local Fourier uniformity of divisor function on average]{Local Fourier uniformity of higher divisor functions on average}

\author{Mengdi Wang}

\address{Department of Mathematics and Statistics, University of Turku, Turku 20014, Finland}
\email{mengdi.wang@utu.fi}


\maketitle

\begin{abstract}
Let $\tau_k$ be the $k$-fold divisor function. By constructing an approximant of $\tau_k$, denoted as $\tau_k^*$, which is a normalized truncation of the $k$-fold divisor function, we prove that when $\exp\left(C\log^{1/2}X(\log\log X)^{1/2}\right)\leq H\leq X$ and $C>0$ is sufficiently large,  the following estimate holds for almost all $x\in[X,2X]$:
\[
\Big|\sum_{x<n\leq x+H}(\tau_k(n)-\tau_k^*(n)) e(\alpha_dn^d+\cdots+\alpha_1n)\Big|=o(H\log^{k-1}X),
\]
where $\alpha_1, \dots, \alpha_d\in \mathbb{R}$ are arbitrary frequencies.

\end{abstract}

\section{Introduction}

Denote $\tau_k(n)$ as the $k$-fold divisor function, signifying the number of representations of $n$ as the product $n=n_1\cdots n_k$ of $k$ natural numbers, where $k \geq 2$ is a fixed integer. The Dirichlet hyperbola method leads to the following asymptotic expression 
\[
\sum_{X<n\leq 2 X} \tau_k(n)=XP_k(\log X) +O(X^{1-1/k+\eps}),
\]
where $P_k$ is a polynomial of degree $k-1$ and $\eps>0$ is arbitrarily small. We expect  a similar mean value result for $\tau_k$ within a relatively short interval $[X, X+H]$. For instance, according to \cite[Theorem 1.3]{MT}, when $X^{0.55+\eps}\leq H\leq X$, one has
\[
\sum_{X<n\leq X+H}\tau_k(n)=HP_k(\log X)+O(H\log^{(2/3+\eps)k-1}X).
\]
Meanwhile, in a recent paper \cite{Sun},  it is demonstrated that, assuming $\log^{k\log k-k+1+\eps}X\leq H\leq X$, for almost all integers $x\in(X,2X]$, the following holds
\[
\sum_{x<n\leq x+H}\tau_k(n)=C_kH\log^{k-1}X+o(H\log^{k-1}X),
\] 
where $C_k>0$ is a constant.
 Building on these observations, we say a 1-bounded function $g:\N\to\C$  is \textit{orthogonal} to  $\tau_k$ in the interval $[X,X+H]$ if
\[
\sum_{X<n\leq X+H}\tau_k(n) g(n)=o(H\log^{k-1}X);
\]
otherwise, we say $\tau_k$ and $g$ are \textit{correlated}  in the interval $[X,X+H]$.

Assume $e(n\alpha)$ denotes a major-arc phase. It is known that $\tau_k$ and $e(\cdot\alpha)$ are correlated in $[X,2X]$. To address this and establish the orthogonality of the divisor function with all phases, we can either transfer our focus to consider the $W$-tricked divisor functions, as demonstrated in \cite[Theorem 9.1]{Mat} (for $k=2$); or, as in \cite{MRSTT,MSTT}, subtract off a suitable approximant from $\tau_k$.

In the series of papers \cite{MRSTT,MSTT}, Matom{\"a}ki, Radziwi{\l}{\l},  Shao,  Tao and Ter{\"a}v{\"a}inen construct an approximant $\tau_k^\#$ (defined in \cite[(1.2)]{MSTT}) and show that $\tau_k-\tau_k^\#$ is orthogonal to degree $d$ nilsequences in all short intervals and almost all very short intervals. If we focus solely on polynomial phase functions, \cite[Theorem 1.1]{MSTT} demonstrates that when $X^{3/5+\eps}\leq H\leq X$, we have
\[
\sup_{\alpha_1,...,\alpha_d\in\T}\Bigabs{\sum_{X<n\leq X+H}(\tau_k(n)-\tau_k^\#(n)) e(\alpha_dn^d+\cdots+\alpha_1n)}\ll H\log^{3k/4-1}X,
\]
and when $k=2$ the above $H$ can be as small as $X^{1/3+\eps}$; and \cite[Theorem 1.1]{MRSTT} shows that if $X^\eps\leq H\leq X$ then  
\begin{align}\label{mrstt}
\sup_{\alpha_1,...,\alpha_d\in\T}\Bigabs{\sum_{x<n\leq x+H}(\tau_k(n)-\tau_k^\#(n)) e(\alpha_dn^d+\cdots+\alpha_1n)}=o(H\log^{k-1}X)	
\end{align}
holds for  all but $O(X\log^{-A}X)$ integers $x\in[X,2X]$.

 Suppose that $n\in[X,2X]$ and $0<\gamma<\frac{1}{5k}$  is a constant. Let's define
\[
\tau_k^*(n)=\gamma^{1-k}\twosum{m|n}{m\leq X^\gamma}\tau_{k-1}(m).
\]
The approximant\footnote{We would refer the readers to \cite[Lemma 6.1]{MRT-divisor} and \cite[Section 4]{Mat} for the construction of a majorant rather than approximant, and also refer to \cite[Section 3.1]{MSTT} for the discussion of distinct approximants.} $\tau_k^*$ can be viewed as a normalized truncation of $\tau_k$, and this construction is inspired by the minorant in \cite{SA} but with slight modifications.  Our main result is as follows.

\begin{theorem}\label{main}
	Suppose that $0<\eta<1$ is a parameter and $X\geq 1$ is sufficiently large, there is a large constant $C=C(\eta)>0$ such that when $\exp\bigbrac{C\log^{1/2}X(\log\log X)^{1/2}}\leq H\leq X$ we have
	\[
	\int_X^{2X}\sup_{\alpha_1,\dots,\alpha_d\in\T}\Bigabs{\sum_{x\leq n\leq x+H}\bigbrac{\tau_k(n)-\tau_k^*(n)}e(\alpha_dn^d+\cdots+\alpha_1n)}\,\rd x\leq \eta XH\log^{k-1}X.
	\]
\end{theorem}
 Theorem \ref{main} suggests that, by replacing the approximant with ours,  one can extend the estimate (\ref{mrstt}) to the range $\exp\bigbrac{C\log^{1/2}X(\log\log X)^{1/2}}\leq H\leq X$ for all but $o(X)$ integers $x\in[X,2X]$.

\subsection*{Outline of the proof}

Our proof draws inspiration from the exploration of the local Fourier uniformity conjecture. Suppose that $\lambda$ is the Liouville function. \cite[Conjecture 1]{MRT20} predicts that if $H\to_X\infty$, then
\[
\int_X^{2X}\sup_{\alpha\in\T}\bigabs{\sum_{x<n\leq x+H}\lambda(n) e(n\alpha)}\,\rd x=o(XH).
\]
This conjecture has been investigated in papers  \cite{MRT20, MRTTZ, Wal, Wal23-1, Wal23-2}, and the current best-known result confirms its validity for $\exp\left(C\log^{1/2}X(\log\log X)^{1/2}\right)\leq H\leq X$ by Walsh \cite{Wal23-1}. Moreover, under the assumption of the Generalized Riemann Hypothesis, Walsh \cite{Wal23-2} verifies its effectiveness  when $H\geq(\log X)^{\psi(X)}$ where $\psi(X)$ tends to infinity arbitrarily slowly.


In the following, we will present a brief outline of our proof, omitting several challenges we encountered. One notable challenge is that in the proof, for distinct purposes, we need to shift and shrink the supported intervals. However, especially when dealing with very short intervals,  it becomes difficult to shift and shrink the supported intervals while maintaining local correlation information.

Following the ideas in studying the local Fourier uniformity conjecture, let's assume
\[
\int_X^{2X}\sup_{\alpha\in\T}\bigabs{\sum_{x<n\leq x+H}(\tau_k(n)-\tau_k^*(n)) e(n\alpha)}\,\rd x\gg XH\log^{k-1}X.	
\]
The discretization process allows us to identify $\gg X/H$ of $H$-separated integer $x\in(X,2X]$ such that, for each of these $x$, there exists a corresponding frequency $\alpha_x$ where the local correlation holds,
\begin{align}\label{0-0}
\bigabs{\sum_{x<n\leq x+H} (\tau_k(n)-\tau_k^*(n)) e(n\alpha_x)}\gg H\log^{k-1}X.
\end{align}
We then apply Ramar\'e's identity, serving a similar purpose as Elliott's inequality in the works \cite{MRT20, MRTTZ, Wal, Wal23-1, Wal23-2},  to extract a small prime factor $p\in[P/2,P]$ with $P\approx H^\eps$ from almost all $n\in(x,x+H]$. More or less, we can assume that
\[
\Bigabs{\sum_{\frac{x}{p}<n\leq \frac{x+H}{p}} \frac{\tau_k(n)-\tau_k^*(n)}{\omega_{(P^\eps,P]}(n)+1} e(n(p\alpha_x))}\gg\frac{H\log^{k-1}X}{p},
\]
where $\omega_{(P^\eps,P]}(n)$ counts the number of prime factors of $n$ in the interval $(P^\eps,P]$.
Because $x$ lies in the dyadic interval $[X,2X]$ and $p$ is within the dyadic interval $[P/2,P]$, varying both $p$ and $x$ allows us to predict substantial overlaps in the intervals $[\frac{x}{p},\frac{x+H}{p}]$. Let's say, $I\subseteq[X/P,(4X)/P]$ is an interval of length $|I|\asymp H/P$ and there exist numerous pairs $(p, (x, \alpha_x))$ such that
\begin{align}\label{0-1}
	\Bigabs{\sum_{n\in I}  \frac{\tau_k(n)-\tau_k^*(n)}{\omega_{(P^\eps,P]}(n)+1} e(n(p\alpha_x))}\gg\frac{H\log^{k-1}X}{P}.
\end{align}
We expect that there should only have a few frequencies dominates the above exponential sum, which means that there should be many pairs $(p,(x,\alpha_x))$ and $(q,(y,\alpha_y))$ such that
\begin{align}\label{0-2}
p\alpha_x\approx q\alpha_y \mod 1\quad \text{ and } \quad x/p\approx y/q.
\end{align}
To make the information we obtained about $\alpha_x$ valuable, it is essential to upgrade the congruence relation held $\mod{p'}$ for  large primes $p' \in [P^{\eps^2}, P^\eps]$ rather than $\mod 1$. Following the argument in \cite[Proposition 3.2]{MRT20}, we successfully deduce that
\begin{align}\label{p'-2}
p\gamma_{z_1}\approx q\gamma_{z_2} \mod{p'}\quad \text{ and } \quad z_1/p\approx z_2/q,
\end{align}
for many pairs $(z_1, \gamma_{z_1}), (z_2, \gamma_{z_2}) \in [X/(4P^{1+\eps}), 4X/P^{1+\eps}]\times\T$, $p,q\in[P/2,P]$, and many primes $p' \in [P^{\eps^2}, P^\eps]$. Besides, we also have the local correlation below
\[
\Bigabs{\sum_{z<n\leq z+H/P^{1+\eps}} \frac{\tau_k(n)-\tau_k^*(n)}{(\omega_{(P^\eps,P]}(n)+1)(\omega_{(P^{\eps^2},P^\eps]}(n)+1)}e(n\gamma_z)}\gg(H/P^{1+\eps})\log^{k-1}X.	
\]
Building on the idea of Walsh \cite{Wal23-1}, we can use the information (\ref{p'-2}) to build relatively long paths. This approach effectively establishes a global frequency $\gamma$ with $e(n\gamma)\approx e(an/q)n^{iT}$ for some small $q$ and such that $\gamma\approx\gamma_z$ for many local frequencies $\gamma_z$. To derive the contradiction, it is important to demonstrate that for almost all $z \in [X/P^{1+\eps}, 4X/P^{1+\eps}]$, we have
\[
\Bigabs{\sum_{z<n\leq z+H/P^{1+\eps}} \frac{\tau_k(n)-\tau_k^*(n)}{(\omega_{(P^\eps,P]}(n)+1)(\omega_{(P^{\eps^2},P^\eps]}(n)+1)}e(an/q)n^{iT}}\ll(H/P^{1+\eps})\log^{k-1}X.
\]
However, the function $\omega_{(P^\eps, P]}(n)$ introduced through the use of Ramar\'e's identity is not regular, posing challenges in establishing the above inequality. Therefore, instead of  following the approach outlined in \cite[Proposition 3.2]{MRT20}, we upgrade the congruent in (\ref{0-2}) to $\mod {p'}$ and analyze configurations $(x, \alpha_x) \in [X, 2X]\times \T^d$. By demonstrating the existence of numerous $\alpha_x$ satisfying $e(n\alpha_x)\approx e(an/q)n^{iT}$ for a small $q$ and extending \cite[Theorem 9.2]{MRII} to the divisor function, we can derive a contradiction with (\ref{0-0}).

Finally, let's redirect to  (\ref{0-1}). Observing that within such a short interval direct handling of the exponential in terms of $\tau_k$ becomes challenging, we intend to eliminate the weight $\tau_k$ at a notable cost.  However noting that the second moment of divisor function causes  logarithmic power loss, i.e. 
\[
X^{-1}\sum_{n\sim X}\tau_k^2(n)\gg \log^{k^2-1}X\gg \log^{k^2-2k+1} X \bigbrac{X^{-1}\sum_{X<n\leq 2X} \tau_k(n)}^2,
\]
 the approach outlined in \cite[Lemma 2.4]{MRT20} cannot be applied here. To navigate this obstacle, we address the local divisor correlations $\sum_{x<n\leq x+H}\tau_k(n)\tau_k(n+h)$ with the assistance of \cite{MRT-divisor}.



\subsection*{Acknowledgement}

The author would like to thank Kaisa Matom\"aki, Yu-Chen Sun, and Joni Ter\"av\"ainen for their valuable discussions and suggestions. The author also thanks Yang Cao for helping create the diagrams. This work was supported by the Academy of Finland grant no. 333707.


\section{Notations and preliminaries}

This section serves as introducing fundamental notations and collecting essential results concerning $k$-fold divisor function and frequencies that we required in the paper. Suppose that $2Q_0<Q$ are two numbers, denote $\mathcal  P(Q_0,Q)$ as the product of the primes in the interval $(Q_0,Q]$. Let $c$ and $\epsilon$ be small constants, both independent of $\eta$,  and we allow them to vary in distinct occurrence. Besides, we use the vector $\vec{\alpha} = (\alpha^{(d)}, \dots, \alpha^{(1)}) \in \mathbb{T}^d$ to represent a $d$-tuple frequency. 

If $A$ is a set, we use  $1_A$ to denote the indicator of $A$; and if $A$ is a statement, we let $1_A$ denote the number 1 when $A$ is true and $0$ when $A$ is false.


\subsection{Properties of divisor functions}

\begin{lemma}[Shiu's bound]\label{shiu-bound}
Let $A\geq 1$ and $\epsilon>0$ be fixed. Let $X^\eps\leq H\leq X$ and $q\leq H^{1-\eps}$. Let $f$ be a non-negative multiplicative function such that $f(p^l)\leq A^l$ for every prime power $p^l$ and $f(n)\ll_cn^c$ for every $c>0$. Then
\[
\sum_{X<n\leq X+H \atop n\equiv a\mod q} f(n)\ll\frac{H}{\phi(q)\log X} \exp\Bigbrac{\sum_{p\leq 2X,p
\nmid q}\frac{f(p)}{p}}.
\]

\begin{proof}
\cite[Theorem 1]{Shiu}.	
\end{proof}

\end{lemma}

\begin{lemma}[Local uniformity of divisor function]\label{local-average}
Suppose that $\log^{k\log k}X\leq H\leq X$, then  we have
\[
\int_X^{2X}\Bigabs{H^{-1}\sum_{x\leq n\leq x+H}\tau_k(n)-x^{-1}\sum_{x<n\leq 2x}\tau_k(n)}\,\rd x=o( X\log^{k-1}X).
\]
\end{lemma}

\begin{proof}
\cite[Theorem 1.1]{Sun}.
\end{proof}

We'd like to introduce a corollary to this lemma to make the exceptional set result clear.

\begin{corollary}[Exceptional set and pointwise upper bound]\label{exceptional-set}
Let $0<\delta<1$ be a parameter and $\log^{k\log k}X\leq H\leq X$. There is an exceptional set $\mathcal E\subseteq[X,2X]$ with cardinality $\#\mathcal E\ll\delta X$  and satisfing
\begin{enumerate}
	\item $\sum_{x\in\mathcal E}\sum_{x< n\leq x+H}(\tau_k(n)+\tau_k^*(n))\ll\delta X H\log^{k-1}X$ and,
	\item when $x\in[X,2X]\backslash\mathcal E$ we have the pointwise bound 
	 \[
	 \sum_{x< n\leq x+H}(\tau_k(n)+\tau_k^*(n))\ll H\log^{k-1}X.
	 \]
\end{enumerate}
	
\end{corollary}

\begin{proof}
Suppose that $\mathcal E$ is the  set  of integers $x\in(X,2X]$ such that the pointwise bound $\sum_{x<n\leq x+H}\tau_k(n)\ll H\log^{k-1}X$ doesn't hold, it can be deduced from  Lemma \ref{local-average} that  $\#\mathcal E\ll\delta X$ and $\sum_{x\in\mathcal E}\sum_{x<n\leq x+H}\tau_k(n)\ll\delta XH\log^{k-1}X$. Meanwhile, as it can be seen from the notion of the approximant that $\tau_k^*\leq\gamma^{1-k}\tau_k$ pointwise, one thus has
\[
\sum_{x\in\mathcal E} \sum_{x<n\leq x+H}(\tau_k(n)+\tau_k^*(n))\leq \sum_{x\in\mathcal E} \sum_{x<n\leq x+H} (1+\gamma^{1-k})\tau_k(n)\ll\delta XH\log^{k-1}X,
\]
and when $x\in[X,2X]\backslash\mathcal E$ one has
\[
\sum_{x<n\leq x+H}(\tau_k(n)+\tau_k^*(n)) \leq \sum_{x<n\leq x+H}(1+\gamma^{1-k})\tau_k(n)\ll H\log^{k-1}X.
\]
\end{proof}

\begin{lemma}[Local divisor correlations on average]\label{divisor-correlation}
Let $0<\delta<1$ be a parameter and $\log^{10000k\log k}X\leq H\leq X^{1/2-\eps}$. There is an exceptional set $\mathcal E\subseteq[-H,H]$ with cardinality $\#\mathcal E\ll\delta^2 H$ such that the following statement holds for all but $O(\delta X)$  integers $x\in[X,2X]$:
\begin{enumerate}
\item(pointwise divisor correlation) If $h\in[-H,H]\backslash\mathcal E$ we have 
\[
\sum_{x<n\leq x+H}\tau_k(n)\tau_k(n+h)\ll\delta^{-1}H\log^{2k-2}X;
\]
\item (negligible contribution from the exceptional set)
\[
\sum_{h\in\mathcal E}\sum_{x<n\leq x+H}\tau_k(n)\tau_k(n+h)\ll\delta^2 H^2\log^{2k-2}X.
\]	
\end{enumerate}

\end{lemma}

\begin{proof}
Let's fix an integer $0\neq|h|\leq H$ for a moment. When $X+H\leq n\leq 2X-H$, each term $\tau_k(n)\tau_k(n+h)$ occurs $H$ times in the expression
\[
\sum_{X+H<x\leq 2X-H}\sum_{x<n\leq x+H}\tau_k(n)\tau_k(n+h).
\]
Simultaneously, by considering the estimate $\sum_{X<x\leq X+H}\sum_{x<n\leq x+H}\tau_k(n)\tau_k(n+h)\ll H^2X^\eps$, one can deduce from \cite[Theorem 1.1]{MRT-divisor} that for each $0\neq|h|\leq H$  there exists a constant $c_{k,h}>0$ such that
\[
\sum_{0<|h|\leq H} \Bigabs{\sum_{x\sim X}\sum_{x<n\leq x+H}\tau_k(n)\tau_k(n+h)-c_{k,h}XH\log^{2k-2}X}\leq\delta^4XH^2\log^{2k-2}X.
\]
If we set $\mathcal E$ as the set of $|h|\leq H$ satisfying
\[
\sum_{x\sim X}\sum_{x<n\leq x+H}\tau_k(n)\tau_k(n+h)\gg XH\log^{2k-2}X,
\]
it is immediately clear that $\#\mathcal E\ll\delta^4H$. 	Moreover, considering the estimate $\sum_{x\sim X}\tau_k^2(x)\ll X\log^{2k-1}X$, derived from Lemma \ref{shiu-bound}, we can also obtain that
\begin{multline*}
\sum_{x\sim X}\sum_{h\in\mathcal E}\sum_{x<n\leq x+H}\tau_k(n)\tau_k(n+h)\\
\leq \sum_{x\sim X}\sum_{x<n\leq x+H}\tau_k^2(n)	+\sum_{h\in\mathcal E\backslash\set{0}}\sum_{x\sim X}\sum_{x<n\leq x+H}\tau_k(n)\tau_k(n+h)\\
\ll XH\log^{2k-1}X+\delta^4XH^2\log^{2k-2}X\ll\delta^4XH^2\log^{2k-2}X.
\end{multline*}
Therefore, it follows from the pigeonhole principle that for   all but $O(\delta^2 X)$ integers $x\in[X,2X]$ we have
\begin{align}\label{2.5.1}
\sum_{h\in\mathcal E}\sum_{x<n\leq x+H}\tau_k(n)\tau_k(n+h)\ll\delta^2 H^2\log^{2k-2}X.
\end{align}

Now we may assume that $h\not\in\mathcal E$ which means that
\[
\sum_{x\sim X}\sum_{x<n\leq x+H}\tau_k(n)\tau_k(n+h)\ll XH\log^{2k-2}X.
\]
Similarly, pigeonhole principle  indicates that when $x\in[X,2X]$ is outside a set with cardinality $O(\delta X)$ the inequality
\begin{align}\label{2.5.2}
\sum_{x<n\leq x+H}\tau_k(n)\tau_k(n+h)\ll \delta^{-1}H\log^{2k-2}X	
\end{align}
holds. The lemma is established by recognizing that there are, in total, at most $O(\delta X)$ integers $x$ in the interval $(X, 2X]$ for which either (\ref{2.5.1}) or (\ref{2.5.2}) does not hold.
\end{proof}


\subsection{Frequencies and pyramids}

The following lemma is an extension of \cite[Lemma 2.1]{Wal23-1}.

\begin{lemma}[Pair of frequencies]\label{2-1}
Let $\eps,\eps'>0$ be small parameters and $Q\in\N$. Let $\vec\alpha_1, \vec\alpha_2\in\T^d$ be frequencies and $p\neq q$ be primes not dividing $Q$. Suppose that 
\begin{align}\label{relation}
p^j\alpha_1^{(j)} \equiv q^j \alpha_2^{(j)} +O(\eps^j+\eps'^j)\mod Q \text{ for all } 1\leq j\leq d.	
\end{align}
Then there is a frequency $\vec\alpha\in\T^d$ such that
\[
p^j\alpha^{(j)} \equiv \alpha_2^{(j)} +O\bigbrac{\bigbrac{\frac{\eps'}{q}}^j} \mod Q \text{ and } 
q^j\alpha^{(j)} \equiv \alpha_1^{(j)} +O\Bigbrac{\bigbrac{\frac{\eps}{p}}^j} \mod Q. 
\]

\end{lemma}

\begin{proof}
We are going to fix the degree $1\leq j\leq d$ in the proof. If we take $p^j\beta_1^{(j)}\equiv\alpha_2^{(j)} \mod Q$ and $q^j\beta_2^{(j)}\equiv \alpha_1^{(j)}\mod Q$, the assumption (\ref{relation}) implies that
\[
(pq)^j\beta_1^{(j)} \equiv (pq)^j \beta_2^{(j)}+O\bigbrac{\eps^j+\eps'^j}\mod Q.
\]
On the other hand, if it is necessary one can add an integer multiple of $\frac{Q}{p^j}$ to $\beta_1^{(j)}$ and an integer multiple of $\frac{Q}{q^j}$ to $\beta_2^{(j)}$ to assume that
\[
\beta_1^{(j)} \equiv \beta_2^{(j)}+O(\frac{Q}{2p^jq^j})\mod Q.
\]
Combining these two estimates  we  have
\[
\beta_1^{(j)} \equiv \beta_2^{(j)}+O\Bigbrac{\frac{\eps^j+\eps'^j}{p^jq^j}}\mod Q.
\]
Taking $\alpha^{(j)}=\frac{\eps^j}{\eps^j+\eps'^j}\beta_1^{(j)} + \frac{\eps'^j}{\eps^j+\eps'^j}\beta_2^{(j)}$, it is easy to see that 
\[
p^j\alpha^{(j)}-\alpha^{(j)}_2 \equiv \frac{\eps'^jp^j}{\eps^j+\eps'^j} \bigbrac{\beta_2^{(j)}-\beta_1^{(j)}}
\equiv O\Bigbrac{\frac{\eps'^j}{q^j}}\mod Q,
\]
and similarly, $q^j\alpha^{(j)}\equiv \alpha_1^{(j)}+O\bigbrac{\frac{\eps^j}{p^j}} \mod Q$. Therefore, the frequency $\vec\alpha=(\alpha^{(d)},\dots,\alpha^{(1)})$ satisfies the conditions of this lemma.	
\end{proof}

We now explore the situation where there exists an ordered sequence of frequencies, and any two adjacent frequencies in the sequence satisfy condition (\ref{relation}). Let's say, $\vec\alpha_1,\dots,\vec\alpha_{k+1}\in\T^d$ are frequencies, $p_1,\dots,p_k,q_1,\dots,q_k$ are distinct primes not dividing $Q$. Suppose that for every $1\leq i\leq k$ and $1\leq j\leq d$ we have
\begin{align}\label{relation-2}
p_i^j\alpha_i^{(j)} \equiv q_i^j\alpha_{i+1}^{(j)} +O(\eps^j)\mod Q.	
\end{align}
The first application of Lemma \ref{2-1} with $\eps=\eps'$ gives an ordered sequence $\set{ \vec\alpha_{2,1},\dots,\vec\alpha_{2,k}} $ such that
\begin{align}\label{t=2}
p_i^j\alpha_{2,i}^{(j)} \equiv \alpha_{i+1}^{(j)} +O\Bigbrac{(\frac{\eps}{q_i})^j}\mod Q
\text{ and } q_i^j\alpha_{2,i}^{(j)} \equiv \alpha_i^{(j)}+  O\Bigbrac{(\frac{\eps}{p_i})^j}\mod Q.
\end{align}
Taking into account the relations between both $\vec\alpha_{2,i}$ and $\vec\alpha_{2,i+1}$ with $\vec\alpha_{i+1}$ through the aforementioned congruence condition, it can be demonstrated, using the triangle inequality, that $\vec\alpha_{2,i}$ is also related to $\vec\alpha_{2,i+1}$ by the  congruence equation below
\[
p_i^j\alpha_{2,i}^{(j)} \equiv q_{i+1}^j\alpha_{2,i+1}^{(j)}+O\Bigbrac{\frac{\eps^j}{q_i^j}+\frac{\eps^j}{p_{i+1}^j}} \mod Q,
\]
and Lemma \ref{2-1} can be utilized once more. One can repeatedly apply Lemma \ref{2-1} to build a pyramid.

\begin{definition}[Pre-paths and pyramids]\label{pyramid}
Suppose that $Q\in\N$. We say that a parameter $0<\eps<1$, a $2k$-tuple distinct primes $p_1,\dots,p_k,q_1,\dots,q_k$ not dividing $Q$ and an ordered $(k+1)$-tuple $(\vec\alpha_1,\dots,\vec\alpha_{k+1})$ together 	form a \textit{pre-path $\mod Q$ of length $k$} if the tuples satisfy (\ref{relation-2}). We also say that the initial sequence $\{\vec\alpha_1,\dots,\vec\alpha_{k+1}\}$ and the corresponding sequence $\set{\vec\alpha_{t,i}}_{2\leq t\leq k+1,1\leq i\leq k+2-t}$ obtained in the previous paragraph collectively form a \textit{pyramid}, with $\vec\alpha_{k+1,1}$ referred to as the \textit{top element}. 
\end{definition}

It's worth noting that this definition slightly differs from the one in \cite[Definition 2.3]{Wal23-1}, where it was initially introduced. 
$$
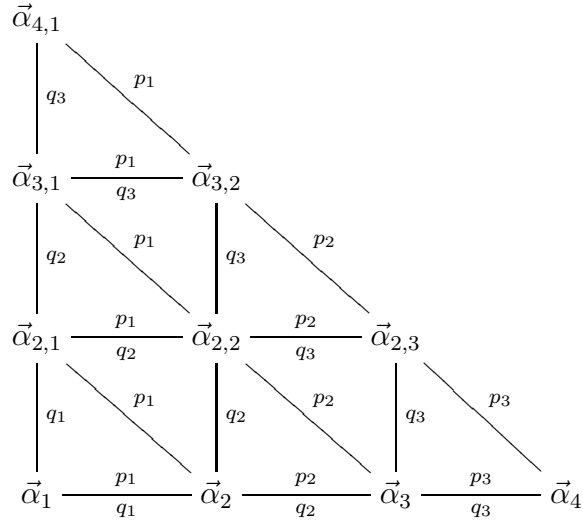
\xymatrix@=8ex{
\vec\alpha_{4,1} &\\
\vec\alpha_{3,1}  \ar@{-}[r]^{p_1}_{q_3}  \ar@{-}[u]_{q_3}  & \vec \alpha_{3,2}  \ar@{-}[lu]_{p_1} &\\
\vec\alpha_{2,1}  \ar@{-}[r]^{p_1}_{q_2}  \ar@{-}[u]_{q_2}  &  \vec\alpha_{2,2}  \ar@{-}[r]^{p_2}_{q_3}  \ar@{-}[u]_{q_3} \ar@{-}[lu]_{p_1} & \vec\alpha_{2,3} \ar@{-}[lu]_{p_2} &\\
\vec\alpha_{1}\ar@{-}[r]^{p_1}_{q_1}\ar@{-}[u]_{q_1} & \vec\alpha_{2} \ar@{-}[r]^{p_2}_{q_2}\ar@{-}[u]_{q_2}  \ar@{-}[lu]_{p_1} & \vec\alpha_{3} \ar@{-}[r]^{p_3}_{q_3}\ar@{-}[u]_{q_3}  \ar@{-}[lu]_{p_2} &\vec\alpha_{4}\ar@{-}[lu]_{p_3}  \\
 &&  
}$$ The diagram above illustrates a pyramid of length 3.

For a pre-path $\mod Q$ of length $k$, we can determine the congruence relations among any two elements in the pyramid. Nevertheless,  we are particularly interested in the congruence relations between the top element and elements in the initial sequence.

\begin{lemma}\label{top-element}
Let $Q\in\N$, $\set{\eps,p_1,\dots,q_k,(\vec\alpha_1,\dots,\vec\alpha_{k+1})}$	form a pre-path $\mod Q$ of length $k$, and $\vec\alpha$ be the top element of the pyramid. Suppose that $\prod_{i_1\leq i\leq i_2}p_i\asymp\prod_{i_1\leq i\leq i_2}q_i$ whenever $1\leq i_1< i_2\leq k$, then for all $0\leq i'\leq k$ and $1\leq j\leq d$ we have
\[
\prod_{i\leq i'}p_i^j\prod_{i'< i\leq k}q_i^j \alpha^{(j)} \equiv \alpha_{i'+1}^{(j)}+O(k\eps^j)\mod Q.
\]
\end{lemma}

\begin{proof}

For simplicity, we may represent $\vec\alpha_{i}$ as $\vec\alpha_{1,i}$, and denote the top element $\vec\alpha$ as $\vec\alpha_{k+1,1}$ in the proof.
Our approach follows the principle of the left lower triangular, signifying that we exclusively utilize relations in the lower directions first and then employ relations in the left-lower directions. Assuming $2\leq t\leq k+1$ and $1\leq i\leq k+2-t$, we assert that for all $1\leq j\leq d$, the following two congruence equations hold:
\begin{align}\label{rule-1}
q^j_{i+t-2}\alpha_{t,i}^{(j)} \equiv \alpha_{t-1,i}^{(j)}+\Bigbrac{\frac{\eps^j}{(p_i\dots p_{i+t-2})^j}}\mod Q	
\end{align}
and
\begin{align}\label{rule-2}
p^j_i\alpha_{t,i}^{(j)}\equiv \alpha_{t-1,i+1}^{(j)}+O\Bigbrac{\frac{\eps^j}{(q_i\dots q_{i+t-2})^j}} \mod Q.	
\end{align}
This claim can be easily confirmed through induction on $t$. When $t=2$ it follows from (\ref{t=2}). Now we suppose the claim holds for $2\leq t\leq k$ and examine the case for $t+1$. Utilizing (\ref{rule-1}) with $i$ replaced by $i+1$ and keeping (\ref{rule-2}) in mind, the triangle inequality demonstrates that
\[
q^j_{i+t-1}\alpha_{t,i+1}^{(j)} \equiv p_i^j\alpha_{t,i}^{(j)}+O\Bigbrac{\frac{\eps^j}{(p_{i+1}\dots p_{i+t-1})^j}+\frac{\eps^j}{(q_i\dots q_{i+t-2})^j}} \mod Q.
\]
The application of Lemma \ref{2-1} with $\eps=\frac{\eps}{p_{i+1}\dots p_{i+t-1}}$ and $\eps'=\frac{\eps}{q_i\dots q_{i+t-2}}$ then gives a frequency $\alpha_{t+1,i}^{(j)}$ such that
\[
q^j_{i+t-1}\alpha_{t+1,i}^{(j)} \equiv \alpha_{t,i}^{(j)}+O\Bigbrac{\frac{\eps^j}{(p_{i}\dots p_{i+t-1})^j}}\mod Q
\]
and \[
p_i^j\alpha_{t+1,i}^{(j)} \equiv \alpha_{t,i+1}^{(j)}+O\bigbrac{\frac{\eps^j}{(q_i\dots q_{i+t-1})^j}}\mod Q.
\]

Let's redirect our focus to the lemma itself. First of all, we aim to establish the relations between $\vec\alpha$ and the elements that share the same vertical line with $\vec\alpha$. Assuming $0\leq i'\leq k$, iteratively applying (\ref{rule-1}) with $i=1$ and considering the given assumption $\prod_{i_1\leq i\leq i_2}p_i\asymp\prod_{i_1\leq i\leq i_2}q_i$ one can get that
\begin{multline*}
\prod_{i'< t\leq k}q^j_t\alpha^{(j)} \equiv \prod_{i'< t\leq k-1}q_t^j\Bigbrac{\alpha_{k,1}^{(j)}+O\Bigbrac{\frac{\eps^j}{(p_1\dots p_k)^j}}}\\
	\equiv \prod_{i'< t\leq k-1}q_t^j\alpha_{k,1}^{(j)}+O\Bigbrac{\frac{\eps^j}{\prod_{t\leq i'}p_t^j}}\\
	\equiv\cdots\equiv \alpha_{i'+1,1}^{(j)}+O\Bigbrac{(k-i')\frac{\eps^j}{\prod_{t\leq i'}p_t^j}}\mod Q.
\end{multline*}
Next, we consider the situation involving elements in the horizontal line, namely, $\vec\alpha_{1,1},\dots,\vec\alpha_{1,k+1}$. Suppose that $0\leq i'\leq k$, by employing (\ref{rule-2}), we can establish the relation between $\vec\alpha_{1,i'+1}$ and $\vec\alpha_{i'+1,1}$.
\begin{multline*}
\prod_{i\leq i'}p_i^j\alpha_{i'+1,1}^{(j)}\equiv \prod_{i\leq i'-1}p_i^j\Bigbrac{\alpha_{i',2}^{(j)}+O\Bigbrac{\frac{\eps^j}{(q_1\dots q_{i'})^j}}}\\
\equiv	\prod_{i\leq i'-1}p_i^j\alpha_{i',2}^{(j)}+O(\eps^j)\equiv\cdots \equiv \alpha_{1,i'+1}^{(j)}+O(i'\eps^j)\mod Q.
\end{multline*}
Therefore, combining the above two estimates together, one has
\[
\prod_{i\leq i'}p_i^j\prod_{i'< i\leq k}q_i^j \alpha^{(j)} \equiv \prod_{i\leq i'}p_i^j\Bigbrac{\alpha_{i'+1,1}^{(j)}+O\Bigbrac{(k-i')\frac{\eps^j}{\prod_{t\leq i'}p_t^j}}}\equiv 	\alpha_{1,i'+1}^{(j)}+O(k\eps^j)\mod Q,
\]
this concludes the lemma.

\end{proof}


\section{Major arc estimates}

We are going to concentrate on integers having both small prime factors and also large prime factors in this section. In practice, suppose that $2\leq h\leq X$ and $\delta=\eta^A$ for some large constant $A>0$. Let $P_1=h^{2\delta^3}$, $Q_1=h^{\delta^2}$, $P_2=h^{100\delta^2}$ and $Q_2=h^\delta$. Additionally, following the notation used in \cite[page 90]{MRTTZ}, set $v_1=\delta^2/4000$, $v_2=\frac{1}{10}$, $P_3=X^{v_1}$, $Q_3=P_4=X^{\sqrt{v_1v_2}}$, and $Q_4=X^{v_2}$. Now, let $S\subseteq(X,2X]$ denote the set of integers $n$ with a prime factor in each interval $(P_i,Q_i]$ for $1\leq i\leq4$, as well as the total number of prime factors of $n$ does not exceed $(k+\eps )\log\log X$. That is,
\begin{align}\label{s-3}
S=\set{n\in(X,2X]:\Omega(n)\leq (k+\eps )\log\log X, \quad(n,\mathcal P(P_i,Q_i))>1\text{ for all }1\leq i\leq 4}.
\end{align}

\begin{remark}
The set $S$ we defined above allows us to draw an analogous conclusion to \cite[Theorem 9.2]{MRII} for the function $\tau_k1_S$. The assumptions that $n$ has prime factors in the intervals $(P_1,Q_1]$ and $(P_2,Q_2]$, along with the condition $\Omega(n)\leq (k+\eps )\log\log X$, enable us to achieve a M\"atom\"aki-Radziw{\l}{\l} type conclusion for the typical divisor function. To ensure that the exceptional set is extremely small, we need to adapt \cite[Proposition 8.3]{MRII}. Consequently, we also assume that $n$ has prime factors in the intervals $(P_3,Q_3]$ and $(P_4,Q_4]$. It's worth mentioning that the assumption $\Omega(n)\leq (k+\eps )\log\log X$ is not strictly necessary here because our parameter $H$ is not as small as in \cite[Theorem 1.3]{Sun}. However, to directly apply the arguments from \cite{Sun} without any modifications, we include this assumption in our typical set  $S$. 
\end{remark}

\begin{proposition}[Major-arc approximation]\label{major-arc-approx}
Let  $\log^{-1/1000k}X<\delta<\eta<1$ be two parameters. Suppose that $d\in\N$ is an integer, $\exp(\log^{1/2}X)\leq h\leq X^{1/2-\eps}$, $|T^*|\leq X^{d+1}$, $S\subseteq(X,2X]$ is  as in (\ref{s-3}), and $1\leq a\leq q\ll1$ are integers, then we have
\[
\Biggabs{\twosum{x<n\leq x+h}{n\equiv a\mod q}1_S(n)\tau_k(n)n^{iT^*}-\twosum{x<n\leq x+h}{n\equiv a\mod q}\tau_k^*(n)n^{iT^*}}\ll \delta h\log^{k-1}X
\]
for all but $O(X(h^{-\delta^3/50}+X^{-\frac{\delta^6}{10^{16}}}))$ integers $x\in(X,2X]$.
	
\end{proposition}

As in  \cite[Section 3]{MRSTT} and \cite[Section 3]{MSTT}, we are going to introduce long averages into the above left-hand side expression so that Proposition \ref{major-arc-approx} can be simplified into proving the following three assertions.

\begin{proposition}[Long and short averages of $\tau_k$]\label{long-short-divisor}
Let $\delta,T^*,S,q$ be as in Proposition \ref{major-arc-approx} and $\chi\mod q $ be a Dirichlet character.  Then when $\exp(\log^{1/2}X)\leq h\leq X^{1/2-\eps}\leq h_1 \leq X^{1-\frac{1}{100k}}$ one of the following statements holds for  all but $O(X(h^{-\delta^3/50}+X^{-\frac{\delta^6}{10^{16}}}))$ integers $x\in(X,2X]$:
\begin{enumerate}
\item if $\chi\neq\chi_0$ or $|T^*|>X$ then
	\[
	\bigabs{ \sum_{x<n\leq x+h\atop n\in S}\tau_k(n)\chi(n)n^{iT^*}}\ll\delta h\log^{k-1}X;
	\]
	\item otherwise we have
\[
	\biggabs{h^{-1} \sum_{x<n\leq x+h\atop n\in S}\tau_k(n)\chi_0(n)n^{iT^*}-h^{-1}\int_x^{x+h} u^{iT^*}\,\rd u\cdot h_1^{-1}\sum_{x<n\leq x+h_1}\tau_k(n)\chi_0(n)}	 \ll\delta \log^{k-1}X.
\]
\end{enumerate}

\end{proposition}

\begin{proof}
Let $|t_0|\leq X$ be the number $t\in[-X,X]$ that minimizes the distance
\[
\sum_{p\leq 2X }\frac{\tau_k(p)-\Re \tau_k(p)\chi(p)p^{iT^*-it}}{p}.
\]
First of all, we aim to prove that for all but at most $O(X(h^{-\delta^3/50}+X^{-\frac{\delta^6}{10^{16}}}))$ integers $x\in(X,2X]$ the following inequality holds
\begin{align} \label{all-character}
\biggabs{h^{-1} \sum_{x<n\leq x+h\atop n\in S}\tau_k(n)\chi(n)n^{iT^*}-\frac{\int_x^{x+h} u^{it_0}\,\rd u}{h}\cdot h_1^{-1}\sum_{x<n\leq x+h_1\atop n\in S}\tau_k(n)\chi(n)n^{i(T^*-t_0)}}	\leq \delta\log^{k-1}X.
	\end{align}
The proof essentially follows from \cite[Theorem 9.2]{MRII},  but it is somewhat simpler since our $h_1$ still maintains some distance from $X$. We will provide a brief proof with a particular focus on adapting the argument from \cite[Theorem 9.2]{MRII} to our cases with the assistance of \cite{Sun}. Taking $T_0=\log^{1/500k}X, P'=P_3, Q'=Q_3, P=P_4$ and $Q=Q_4$ in \cite[Proposition 8.3]{MRII}, considering the definition of $R_C(1+it)$ in \cite[(44)]{MRII}, the application of \cite[Proposition 8.3]{MRII} shows that
\begin{multline*}
h^{-1} \sum_{x<n\leq x+h\atop n\in S}\tau_k(n)\chi(n)n^{iT^*}-\frac{\int_x^{x+h} u^{it_0}\,\rd u}{h}\cdot h_1^{-1}\sum_{x<n\leq x+h_1\atop n\in S}\tau_k(n)\chi(n)n^{i(T^*-t_0)}\\
= A(x,h,h_1,t_0,T_0,\mathcal U)+O_\eps\Bigbrac{\log^{k-1}X\Bigbrac{\frac{h_1T_0^2}{X}+\frac{1}{h}+\frac{(1+C\log\delta^{-1})^2}{T_0^{10}}}}\\
+O\biggbrac{(1+C\log\delta^{-1})^2\sup_{|t|\leq X/2, |t-t_0|\geq  T_0 \atop \frac{X}{16Q_3Q_4}\leq C\leq \frac{16X}{P_3P_4}}\biggabs{\sum_{m\sim C}\frac{f(m)}{m^{1+it}(\omega_{(P_3,Q_3]}(m)+1)(\omega_{(P_4,Q_4]}(m)+1)}}},	
\end{multline*}
 where $f(m)=\tau_k(m)\chi(m)m^{iT^*}$ and $\omega_{(P_i,Q_i]}(m)$ counts the prime factors of $m$ in the interval $(P_i,Q_i]$.
The second term is obviously bounded by $O(\log^{k-1-1/250k}X)$ from the assumptions of $h$, $h_1$ and $T_0$, and this is acceptable. Meanwhile, taking note that  \cite[Lemma 3.3]{Sun} holds whenever $f$ is a $\tau_k$-bounded multiplicative function,  we can estimate the above third term as
\[
\ll (1+C\log\delta^{-1})^{4k} \log^{k-1}X \bigbrac{T_0^{-1/2} + \frac{\log\log X}{(\log X)^\rho}}.
\]
Here, from \cite[Lemma 3.1]{Sun}, $\rho$ can be chosen as 
\[
\rho=\frac{1}{2}\Bigbrac{\frac{k\alpha}{3}-\frac{2k}{3\pi}\sin\bigbrac{\frac{\pi\alpha}{2}}}
\] 
and $\alpha$ is defined in \cite[Definition 1.2]{Sun}. Note that in our case, we can set the parameter $\alpha$ in \cite[Definition 1.2]{Sun} to $1/2$, which allows us to choose $\rho$ as $1/20$. Consequently, the above estimate can be bounded by $O(\delta \log^{k-1}X)$.

In light of the pigeonhole principle, it suffices to prove that
\[
\frac{1}{X}\int_X^{2X}|A(x,h,h_1,t_0,T_0,\mathcal U)|^2\,\rd x\ll h^{-\frac{\delta^3}{50k}}+X^{-\frac{\delta^6}{10^{16}}}.
\]

It follows from \cite[(46)]{MRII} that
\begin{multline*}
\frac{1}{X}\int_X^{2X}|A(x,h,h_1,t_0,T_0,\mathcal U)|^2\,\rd x\ll\max_{X/h\leq T\leq X}\frac{X/h}{T}	\int_{[-T,T]\backslash\mathcal U}\Bigabs{\sum_{n\sim X\atop n\in S}\frac{\tau_k(n) \chi(n) n^{iT^*}}{n^{1+it}}}^2\,\rd t\\
+\frac{\log^{k^2-1}X}{h}+\frac{\log^5X}{X^{(\delta/20)^6/160}} \max_{X/h\leq T\leq X} \Bigbrac{\frac{\bigabs{\bigbrac{ \mathcal U +[-X^\eps,X^\eps]}\cap[-T,T]}\cdot Q_3}{hT^{1/2}}+1},
\end{multline*}
since the second term is obviously negligible, it can be simplified to prove the following two events:
 \begin{enumerate}
	\item $\mathcal U\subseteq(X,2X]$ is a set such that for any $T\in[X/h,X]$ we have
\[
\bigabs{\bigbrac{ \mathcal U +[-X^\eps,X^\eps]}\cap[-T,T]}\ll\frac{hT^{1/2}}{Q_3};	 
\]
\item The following estimate holds uniformly for $X/h\leq T\leq X$
\[
\int_{[-T,T]\backslash\mathcal U}\Bigabs{\sum_{n\sim X\atop n\in S}\frac{\tau_k(n) \chi(n) n^{iT^*}}{n^{1+it}}}^2\,\rd t\ll P_1^{-1/100}\frac{T\log^{2k-2}X}{X/h}.
\]
\end{enumerate}
Given a number $X/h\leq T\leq X$, let's assume that $\mathcal T_1$, $\mathcal T_2$, and $\mathcal U$ are disjoint sets with a union of $[-T,T]$. The precise definitions are in accordance with the proof of \cite[Proposition 4.4]{Sun}, with parameters $P_1$, $P_2$, $Q_1$, and $Q_2$ substituted as in our case. It follows from the assumption of $\mathcal U$ that if $t\in\mathcal U$, there exists some $v\in\mathcal I_2$ such that $|Q_{v,2}(1+it_0+it)|>e^{-\alpha_2v/P_1^{1/6}}$ with $\alpha_2=\frac{1}{4}-\frac{1}{100}$, then  \cite[Lemma 4.5 (ii)]{Sun} reveals that
\[
|\mathcal U|\ll|\mathcal I_2|T^{2\alpha_2}Q_2^{2\alpha_2}\exp\bigbrac{2k\frac{\log X\log_2X}{\log P_2}}\ll hT^{1/2}(\frac{h}{X})^{1/50} h^{-1+\frac{\delta}{2}}P_1^{1/6}.
\]
The first claim follows by noting that $h\leq X^{1/2-\eps}$ and $Q_3=X^{\delta/200}$. And the second claim follows from the estimates of $E_1$ and $E_2$ on \cite[pages 25-26]{Sun}. Thus, we have finished the proof of (\ref{all-character}) so far.

We can actually remove the restriction $n\in S$ from the long averages in (\ref{all-character}) by observing, from Lemma \ref{shiu-bound}, that
\[
\Bigabs{\sum_{x<n\leq x+h_1 \atop n\in S}\tau_k(n)-\sum_{x<n\leq x+h_1}\tau_k(n)}\ll\sum_{P_i<p_i\leq Q_i\atop i=1,2,3,4}\sum_{x<n\leq x+h_1}\tau_k(n)1_{(n,\mathcal P(P_i,Q_i))=1}\ll\delta^kh_1\log^{k-1}X.
\]
Besides, \cite[Lemma 3.6]{MRSTT} also informs us that when $|T^*-t_0|\geq X/h_1^{1-\eps}$, we have
\[
\bigabs{ \sum_{x<n\leq x+h_1} \tau_k(n)\chi(n) n^{i(T^*-t_0)} }\ll\delta h_1.
\]
Furthermore, when  $|T^*-t_0|\leq X/h_1^{1-\eps}$ we can reduce to the case that $T^*-t_0=0$ by decomposing the long intervals into short intervals of length $h_1^{1-2\eps}$. In practice,
\[
\sum_{x<n\leq x+h_1} \tau_k(n)\chi(n) n^{i(T^*-t_0)}=\sum_{x_0}\sum_{x_0<n\leq x_0+h_1^{1-2\eps}}\tau_k(n)\chi(n) n^{i(T^*-t_0)},
\]
where $x_0\in[x,x+h_1)$ takes $h_1^{1-2\eps}$-separated values. Since Taylor expansion denominates that $n^{i(T^*-t_0)}=x_0^{i(T^*-t_0)}+O(|T^*-t_0|h_1^{1-2\eps}/x)$ whenever $|n-x_0|\leq h_1^{1-2\eps}$, one thus has
\[
\sum_{x<n\leq x+h_1} \tau_k(n)\chi(n) n^{i(T^*-t_0)}=\sum_{x_0}x_0^{i(T^*-t_0)}\sum_{x_0<n\leq x_0+h_1^{1-2\eps}}\tau_k(n)\chi(n)+O(h_1^{1-\eps}). 
\]
Combining the above analysis, we can therefore conclude that
\begin{multline*}
\biggabs{h^{-1} \sum_{x<n\leq x+h\atop n\in S}\tau_k(n)\chi(n)n^{iT^*}-\frac{\int_x^{x+h} u^{it_0}\,\rd u}{h}\cdot h_1^{-1}\sum_{x_0}x_0^{i(T^*-t_0)}\sum_{x_0<n\leq x_0+h_1^{1-2\eps}}\tau_k(n)\chi(n)}\\
\ll \delta \log^{k-1}X	
\end{multline*}
for all but at most $O(X(h^{-\delta^3/50}+X^{-\frac{\delta^6}{10^{16}}}))$ integers $x\in(X,2X]$.

Our task now reduces to prove that when $\chi \neq \chi_0$ or $\chi=\chi_0$ but $|T^*|>X$ the long average in the left-hand side expression above is always bounded by $O(\delta\log^{k-1}X)$, because when $\chi=\chi_0$ and $|T^*|\leq X$ it is clear from the notion of $t_0$ that $t_0=T^*$. To achieve this, let $h_2=h_1^{1-2\eps}$ and assume $\chi\neq\chi_0$. Perron's formula  leads us to
\[
\sum_{x_0<n\leq x_0+h_2}\tau_k(n)\chi(n)=\frac{1}{2\pi i}\int_{-T}^TL(1+it,\chi)^k\frac{(x_0+h_2)^{1+it}-x_0^{1+it}}{1+it}\,\rd t+O(\frac{x_0^\eps}{T}).
\]
By shifting the contour of the integration to the region $\sigma=1/2$ we get that
\begin{multline*}
	\sum_{x_0<n\leq x_0+h_2}\tau_k(n)\chi(n) =\frac{1}{2\pi i}\int_{-T}^{T}L(1/2+it,\chi)^k \frac{(x_0+h_2)^{1/2+it}-x_0^{1/2+it}}{1/2+it}\,\rd t\\
	\pm \frac{1}{2\pi i}\int_{1}^{1/2} L(u\mp iT,\chi)^k \frac{(x_0+h_2)^{u+iT}-x_0^{u+iT}}{u+iT}\,\rd u +O(\frac{x_0^\eps}{T}).
\end{multline*}
We may apply \cite[Lemmas 10.9-10.10]{Har} and Taylor expansion to bound the vertical integral as follows
\[
\ll \frac{h_2}{x_0^{1/2}}\sup_{|t|\leq T}|L(\frac{1}{2}+it,\chi)|^{\max\set{0,k-4}}\int_{-T}^T|L(\frac{1}{2}+it,\chi)|^4\,\rd t\ll \frac{h_2}{x_0^{1/2}}T^{k/2-1+\eps}\ll h_2^{1-\eps}
\]
if we take $T=X^{1/k}$. Besides, when $k\leq 4$ similar result can be obtained by an application of Cauchy-Schwarz inequality to derive the fourth moment of the $L$-function. Simultaneously, the horizontal line integrals can be estimated as follows:
\[
\ll T^{-1}\int_{1/2}^1 T^{1-u+\eps}x^u\,\rd u\ll x/T \ll h_2,
\]
which is also negligible. 

When $\chi=\chi_0$ and $|T^*|\leq X+X/h_1^{1-\eps}$, it can be seen from the notion of $t_0$ that
\[
t_0=\begin{cases}
X \qquad&\text{if }X<T^*\leq  X+X/h_1^{1-\eps},\\
T^*	\qquad&\text{if } |T^*|\leq X,\\
-X&\text{if } -X-X/h_1^{1-\eps}\leq T^*<-X.
\end{cases}
\]
Besides, as noting that \[
h^{-1} \sup_{|t_0|\geq x}\Bigabs{\int_x^{x+h} u^{it_0}\,\rd u}\ll h^{-1} \sup_{|t_0|\geq x}\bigabs{\frac{(x+h)^{1+it_0}-x^{1+it_0}}{1+it_0}}\ll h^{-1},
\]
which means that when $|t_0|\geq X$ the long average is also negligible. Therefore we can assume that $|T^*|\leq X$ and the proposition follows.

\end{proof}

\begin{proposition}[Long and short averages of $\tau_k^*$]\label{long-short-approx}
 Suppose that the assumptions are as in Proposition \ref{long-short-divisor}. Then when $\exp(\log^{1/2}X)\leq h\leq X^{1/2-\eps}\leq h_1 \leq X^{1-\frac{1}{100k}}$ one of the following statements holds for  all but $O(X(h^{-\delta^3/50}+X^{-\frac{\delta^6}{10^{16}}}))$ integers $x\in(X,2X]$:
\begin{enumerate}

	\item if $\chi\neq\chi_0$ or $|T^*|>X$ then
	\[
	\bigabs{ \sum_{x<n\leq x+h}\tau_k^*(n)\chi(n)n^{iT^*}}\ll\delta h\log^{k-1}X;
	\]
	\item otherwise we have
	\begin{multline*}
	\biggabs{h^{-1} \sum_{x<n\leq x+h}\tau_k^*(n)\chi_0(n)n^{iT^*}-h^{-1}\int_x^{x+h} u^{iT^*}\,\rd u\cdot h_1^{-1}\sum_{x<n\leq x+h_1}\tau_k^*(n)\chi_0(n)}\\	 \ll\delta \log^{k-1}X.
	\end{multline*}
	\end{enumerate}
\end{proposition}

\begin{proof}

Let $t_0$ be as in Proposition \ref{long-short-divisor}, it suffices to prove that for arbitrary Dirichlet character $\chi\mod q$ and any number $|T^*|\leq X^{d+1}$ 
\begin{multline*}
	\biggabs{h^{-1} \sum_{x<n\leq x+h}\tau_k^*(n)\chi(n)n^{iT^*}-h^{-1}\int_x^{x+h} u^{it_0}\,\rd u\cdot h_1^{-1}\sum_{x<n\leq x+h_1}\tau_k^*(n)\chi(n)n^{i(T^*-t_0)}}\\	 \ll\delta \log^{k-1}X
	\end{multline*}
holds for  all but $O(X(h^{-\delta^3/50}+X^{-\frac{\delta^6}{10^{16}}}))$ integers $x\in(X,2X]$, since it can be observed from the proof of Proposition \ref{long-short-divisor} that the long average is always bounded by $O(\delta\log^{k-1}X)$ if $\chi\neq\chi_0$ or $|T^*|\geq X$ and we have $t_0=T^*$ when $\chi=\chi_0$ and $|T^*|\leq X$.

Let's set $D(X;1+it)=\sum_{n\sim X}\frac{\tau_k^*(n)\chi(n)}{n^{1+it}}$ and $W=h_1^{\frac{1}{100k}}$. It follows from Perron's formula that the short average $h^{-1} \sum_{x<n\leq x+h}\tau_k^*(n)\chi(n)n^{iT^*}$ equals to
\[ 
\frac{h^{-1}}{2\pi i} \Bigbrac{\int_{|t-t_0|\leq W} +\int_{|t-t_0|>W}}  D(X;1+i(t-T^*))	\frac{(x+h)^{1+it}-x^{1+it}}{1+it}\,\rd t.	
\]
By making the change of variables $t\to t+t_0$, the first integral is equal to
\[
\frac{h^{-1}}{2\pi i}\int_{|t|\leq W} D(X;1+i(t+t_0-T^*))\int_x^{x+h} u^{i(t+t_0)}\,\rd u\rd t.
\]
Taking note, from  Taylor expansion, that $u^{it}=x^{it}+O(h|t|/x)$ whenever $x\leq u\leq x+h$, this integral then equals to
\[
h^{-1}\int_x^{x+h} u^{it_0}\,\rd u\frac{1}{2\pi i}\int_{|t|\leq W}D(X;1+i(t+t_0-T^*)) x^{it}\,\rd t+O(W^{3/2}h\log^{k^2}X/X),
\]
as from Cauchy-Schwarz inequality and \cite[Lemma 4.3]{Sun} one has
\[
\int_{|t|\leq W}|t||D(X;1+i(t+t_0-T^*)|\,\rd t\ll  W^{3/2}\log^{k^2}X.
\]
 On the other hand, for the long average,  Perron's formula also reveals that
\begin{multline*}
h^{-1}\int_x^{x+h} u^{it_0}\,\rd u\cdot h_1^{-1}\sum_{x<n\leq x+h_1}\tau_k^*(n)\chi(n)n^{i(T^*-t_0)}\\
=\frac{\int_x^{x+h} u^{it_0}\,\rd u}{h}\cdot\frac{h_1^{-1}}{2\pi i}\biggbrac{\int_{|t|\leq W}+\int_{|t|>W}}D(X;1+i(t+t_0-T^*))\frac{(x+h_1)^{1+it}-x^{1+it}}{1+it}\,\rd t.	
\end{multline*}
The application of Taylor expansion yields that the first term equals to
\begin{multline*}
h^{-1}\int_x^{x+h} u^{it_0}\,\rd u	\cdot\frac{1}{2\pi i}\int_{|t|\leq W}D(X;1+i(t+t_0-T^*))x^{it}\frac{(1+h_1/x)^{1+it}-1}{(1+it)h_1/x}\,\rd t\\
=h^{-1}\int_x^{x+h} u^{it_0}\,\rd u	\cdot\frac{1}{2\pi i}\int_{|t|\leq W}D(X;1+i(t+t_0-T^*))x^{it}\,\rd t+O(h^{-\frac{1}{20k}}).
\end{multline*}
Taking note that the main terms within the integral region $|t|\leq W$ in both long and short averages are equal, and the error terms are acceptable, it suffices to show that, by taking the second moment of $x\in[X,2X]$, the integral over $|t|>W$ is acceptable.
\cite[Lemma 8.1]{MRII}  shows that
\begin{multline*}
\frac{1}{Xh^2}\int_X^{2X}\Bigabs{\int_{|t-t_0|>W}D(X;1+i(t-T^*))\frac{(x+h)^{1+it}-x^{1+it}}{1+it}\,\rd t}^2\,\rd x\\
 \ll \max_{T\geq X/h}\frac{X/h}{T}\int_{W<|t-t_0|\leq T}	|D(X;1+i(t-T^*))|^2\,\rd t.
\end{multline*}
As it follows from the large sieve, for example \cite[Lemma 3.4]{MRII}, as well as noting that $\tau_k^*\leq\gamma^{1-k}\tau_k$ pointwise, that when $T\geq X$ one has
\[
\int_{W<|t-t_0|\leq T}	|D(X;1+i(t-T^*))|^2\,\rd t \ll \frac{Th^{-1/2}}{X/h},
\]
thus, it  suffices to prove that
\[
\max_{ X/h\leq T\leq X}\frac{X/h}{T}\int_{W<|t|\leq T}	\Bigabs{\sum_{n\sim X}\frac{\tau_k^*(n)\chi(n) n^{i(T^*-t_0)}}{n^{1+it}}}^2\,\rd t\ll h^{-\delta^3}. 
\]

Now we may recall the definitions of the Dirichlet polynomial $D(X;1+it)$ and the approximant $\tau_k^*$ to see that
\begin{align*}
	|D(X;1+i(t+&t_0-T^*))|\ll\Bigabs{\sum_{m\leq X^\gamma} \frac{\tau_{k-1}(m)\chi(m)m^{i(T^*-t_0)}}{m^{1+it}} \sum_{n\sim\frac{X}{m}}\frac{\chi(n)n^{i(T^*-t_0)}}{n^{1+it}}}\\
	 &\ll \sum_{i\leq\log X^\gamma/\log 2} \biggabs{\sum_{m\sim M_i} \frac{\tau_{k-1}(m)\chi(m)m^{i(T^*-t_0)}}{m^{1+it}}} \biggabs{\sum_{\frac{X}{2M_i}<n\leq \frac{2X}{M_i}} \frac{\chi(n)n^{i(T^*-t_0)}}{n^{1+it}}}.
\end{align*}
If we write $a(m)=\sum_{m_1m_2=m\atop m_1,m_2\sim M_i} \tau_{k-1}(m_1)\tau_{k-1}(m_2)\ll\tau_{3k}(m)$, Cauchy-Schwarz inequality, \cite[Lemma 3.1]{MRII} and \cite[Theorem 1]{BH} together leads us to
\begin{multline*}
\int_{|t|\leq T}	\Bigabs{\sum_{n\sim X}\frac{\tau_k^*(n)\chi(n) n^{i(T^*-t_0)}}{n^{1+it}}}^2\,\rd t\\
\ll\sum_{i\leq\frac{\log X^\gamma}{\log 2}}\biggbrac{\int_{|t|\leq T}\biggabs{\sum_{m\asymp M_i^2} \frac{a(m)\chi(m)m^{i(T^*-t_0)}}{m^{1+it}}}^2\rd t}^{1/2}  \biggbrac{\int_{|t|\leq T} \biggabs{\sum_{\frac{X}{2M_i}<n\leq \frac{2X}{M_i}} \frac{\chi(n)n^{i(T^*-t_0)}}{n^{1+it}}}^4\rd t}^{1/2}	\\
\ll\frac{T\log^AX}{X}\sum_{i\leq \log X^\gamma/\log 2} \Bigbrac{M_i^{-2}\sum_{m\asymp M_i^2}\tau_{3k}^2(m) + M_i^{-2}\sum_{0<|n|\leq M_i^2/T}\sum_{m\asymp M_i^2}\tau_{3k}(m) \tau_{3k}(m+n)}^{1/2} \\
\ll \frac{T\log^AX}{X} \Bigbrac{ \sum_{m\ll X^{2\gamma}}\frac{\tau_{3k}^2(m)}{m}+\sum_{0<|n|\leq X^{2\gamma}/T}\sum_{m\leq X^{2\gamma}}\frac{\tau_{3k}(m) \tau_{3k}(m+n)}{m}}^2\ll\frac{T\log^AX}{X},
\end{multline*}
the last inequality follows from Lemma \ref{shiu-bound}, as well as noting that the first term dominates the expression when $T\geq X/h$ and $X^{2\gamma}<X^{1/5}$.
This completes the proof of this proposition.

\end{proof}

\begin{proposition}[Difference in long intervals]\label{long-difference}
When $1\leq a\leq q\ll1$ and $X^{1-\frac{1}{100k}}\leq h_1\leq X$, we have
\[
\Bigabs{\sum_{X<n\leq X+h_1 \atop n\equiv a\mod q}(\tau_k(n)-\tau_k^*(n))}\ll h_1\log^{k-2}X.
\]
	
\end{proposition}
\begin{proof}
First of all, it can be seen from the notion of $\tau_k^*$ that
\begin{align*}
h_1^{-1}\sum_{X<n\leq X+h_1 \atop n\equiv a\mod q}	\tau_k^*(n) 
&= h_1^{-1} \gamma^{1-k} \sum_{a_1,a_2\mod q\atop a_1a_2\equiv a\mod q}  \sum_{m\leq X^\gamma \atop m\equiv a_1\mod q} \tau_{k-1}(m) \sum_{X<mn\leq X+h_1 \atop n\equiv a_2\mod q}1\\
&= \frac{\gamma^{1-k}}{q} \sum_{a_1,a_2\mod q\atop a_1a_2\equiv a\mod q} \sum_{m\leq X^\gamma \atop m\equiv a_1\mod q} \frac{\tau_{k-1}(m)}{m} +O(X^{-1/3}),
\end{align*}
and similarly, we also have 
\[
X^{-1}\sum_{n \sim X \atop n\equiv a\mod q}	\tau_k^*(n)=\frac{\gamma^{1-k}}{q} \sum_{a_1,a_2\mod q\atop a_1a_2\equiv a\mod q} \sum_{m\leq X^\gamma \atop m\equiv a_1\mod q} \frac{\tau_{k-1}(m)}{m} +O(X^{-1/3}).
\]
Therefore, we can conclude that
\[
h_1^{-1}\sum_{X<n\leq X+h_1 \atop n\equiv a\mod q}	\tau_k^*(n)= X^{-1}\sum_{n \sim X \atop n\equiv a\mod q}	\tau_k^*(n)+O(X^{-1/3}).
\]

On the other hand, by introducing the minorant as in \cite[(1.11)]{SA}, one can also  find that
\begin{align*}
&X^{-1}\sum_{n \sim X \atop n\equiv a\mod q}	\gamma^{1-k}\sum_{m|n \atop m\leq n^\gamma} \tau_{k-1}(m) \\
&= X^{-1}\sum_{n \sim X \atop n\equiv a\mod q}	\gamma^{1-k}\sum_{m|n \atop m\leq X^\gamma} \tau_{k-1}(m) 	+O\Bigbrac{X^{-1}\sum_{m\sim X^\gamma}\tau_{k-1}(m)\sum_{n\sim X/m}1} \\
&= X^{-1}\sum_{n \sim X \atop n\equiv a\mod q}	\tau_k^*(n)+O(\log^{k-2}X).
\end{align*}
Therefore, the application of \cite[Theorem 1.2]{SA} shows that
\begin{align*}
X^{-1}\sum_{n \sim X \atop n\equiv a\mod q}	\tau_k^*(n)
&= X^{-1}\sum_{n \sim X \atop n\equiv a\mod q}	\gamma^{1-k}\sum_{m|n \atop m\leq n^\gamma} \tau_{k-1}(m)+O(\log^{k-2}X)\\
&=X^{-1}\sum_{n\sim X \atop n\equiv a\mod q}\tau_k(n) +O(\log^{k-2}X).	
\end{align*}
 As noted in \cite[Proposition 3.4 (ii)]{MSTT}, when $h_1\geq X^{1-\frac{1}{100k}}$, one has
  \[
 h_1^{-1}\sum_{X<n\leq X+h_1\atop n\equiv a\mod q}\tau_k(n)=X^{-1}\sum_{n\sim X \atop n\equiv a\mod q}\tau_k(n)+O(\log^{-A}X)
 \] 
  we can thus conclude this lemma by making use of the triangle inequality.

\end{proof}

\vspace{2mm}

\noindent\emph{Proof of Proposition \ref{major-arc-approx}.}

\vspace{2mm}
Assume that $h_1=X^{1-\frac{1}{100k}}$ in the proof. On combining Propositions \ref{long-short-divisor}-\ref{long-short-approx}, the triangle inequality shows that when $\chi\neq\chi_0$ or $|T^*|\geq X$ for all but $O(X(h^{-\delta^3/50}+X^{-\frac{\delta^6}{10^{16}}}))$ integers $x\in(X,2X]$ one has
\[
h^{-1}\bigabs{\sum_{x<n\leq x+h \atop n\in S}\tau_k(n)\chi(n)n^{iT^*} - \sum_{x<n\leq x+h }\tau_k^*(n)\chi(n)n^{iT^*}}\ll\delta\log^{k-1}X;
\]
otherwise when $\chi=\chi_0$ and $|T^*|\leq X$ one has
\begin{align*}
	&h^{-1}\bigabs{\sum_{x<n\leq x+h \atop n\in S}\tau_k(n)\chi_0(n)n^{iT^*} - \sum_{x<n\leq x+h }\tau_k^*(n)\chi_0(n)n^{iT^*}} \\
		&\ll\delta\log^{k-1}X + h_1^{-1}\Bigabs{\sum_{x<n\leq x+h_1}\tau_k(n)\chi_0(n)-\sum_{x<n\leq x+h_1}\tau_k^*(n)\chi_0(n)} . 
\end{align*}
	Observing that $\chi(n)=\chi(a)$ if $n\equiv a\mod q$, one may decompose $n$ into arithmetic progressions with common difference $q$ and then it follows from Proposition \ref{long-difference}, as well as the assumption $q\ll1$, that  
	\begin{align*}
	&h^{-1}\bigabs{\sum_{x<n\leq x+h \atop n\in S}\tau_k(n)\chi_0(n) - \sum_{x<n\leq x+h }\tau_k^*(n)\chi_0(n)}\\
	&\ll h_1^{-1}\sum_{a\mod q}\Bigabs{\twosum{x<n\leq x+h_1}{n\equiv a\mod q}(\tau_k(n)-\tau_k^*(n))}+\delta\log^{k-1}X\ll \delta\log^{k-1}X.
	\end{align*}
	Proposition \ref{major-arc-approx} then follows from the orthogonality of the Dirichlet characters. To be precise,
	\begin{align*}
		&\Bigabs{\twosum{x<n\leq x+h}{n\equiv a\mod q} (1_S(n)\tau_k(n)-\tau_k^*(n))n^{iT^*}} \\
		&\leq \frac{1}{\phi(q)}\sum_{\chi\mod q} \Bigabs{\sum_{x<n\leq x+h }(1_S(n)\tau_k(n)-\tau_k^*(n))\chi(n)n^{iT^*} }
		\ll \delta h\log^{k-1}X
	\end{align*}
for all but $O(X(h^{-\delta^3/50}+X^{-\frac{\delta^6}{10^{16}}}))$ integers $x\in(X,2X]$.

\qed

\section{Reduction and discretization}

Suppose that $0<\omega,\eta<1$ and $\omega$ is sufficiently small with respect to $\eta$. Suppose that $Q=H^\omega$ and $Q_0=Q^\omega$, and let $\pq(n)$ denote the number of prime factors of $n$ in the interval $(Q_0, Q]$. Additionally, as $\omega$ is sufficiently small concerning to $\eta$ and recalling the setting of the set $S$ in (\ref{s-3}), we can assume that for all $1\leq i\leq 4$ we have $[Q_0,Q]\cap[P_i,Q_i]=\emptyset$.  To simplify the description of well-spaced intervals, let's introduce the notion of configurations, as defined in \cite[Definition 3.1]{Wal}.

\begin{definition}[Configuration]\label{configuration}
We say $\mathcal J$ is a \textit{$(c,H)$-configuration in $[X,2X]\times\T^d$} if $\#\mathcal J\geq cX/H$ and the first coordinates  are $H$-separated.	
\end{definition}

The objective of this section is to demonstrate that when the integral is large, it is possible to identify numerous well-spaced short intervals wherein the correlation holds among these intervals. Additionally,  we can observe from this section that almost all integers in these short intervals can extract a small prime factor.  This idea also serves as the foundation for the papers \cite{MRT20,MRTTZ,Wal,Wal23-1,Wal23-2}.

\begin{lemma}[Discretization]\label{discretization}
Suppose that $\log^{2k\log k}X\leq H\leq X^{1/2-\eps}$ and $S\subseteq[X,2X]$ is as in (\ref{s-3}) with $0<\delta\leq\eta^2$. Suppose that
\[
\int_X^{2X}\sup_{\vec\alpha\in\T^d} \sup_{|x_0-x|\leq H}\Bigabs{\sum_{x<n\leq x+H} (\tau_k(n)-\tau^*_k(n)) e\bigbrac{\sum_{1\leq j\leq d}\alpha^{(j)}(n-x_0)^j}}\,\rd x\geq\eta XH \log^{k-1}X,
\]
then there is a $(c\eta,H)$-configuration $\mathcal J_{-1}\subseteq[X,2X]\times\T^d$ such that the following three events happen simultaneously.
\begin{enumerate}
\item When $(x,\vec\alpha_x)\in\mathcal J_{-1}$ we have
\[
\sum_{x<n\leq x+H}(\tau_k(n)+\tau_k^*(n))\ll H\log^{k-1}X;
\]
\item when $(x,\vec\alpha_x)\in\mathcal J_{-1}$ there is some $x_0$  satisfying $|x_0-x|\leq H$ and such that

\[
\bigabs{\sum_{x<n\leq x+H}(1_S(n)\tau_k(n)-\tau_k^*(n))e\bigbrac{\sum_{1\leq j\leq d}\alpha_x^{(j)}(n-x_0)^j}}\gg\eta H\log^{k-1}X;
\]
\item suppose that $0\leq h_1,h_2\leq H^{1-\sigma}$ are integers, where  $0\leq\sigma<1/10$ is a constant and, for each $x$  we fix four integers $z_{1,x}^{(1)},z_{2,x}^{(1)},z_{1,x}^{(2)},z_{2,x}^{(2)}$ with $0\leq H^\eps z_{j,x}^{(i)}-x\leq H$, then  when $(x,\vec\alpha_x) \in\mathcal J_{-1}$  we have
\[
\sum_{|n-z_{j,x}^{(i)}|\leq h_i}\tau_k(n)\ll h_i\log^{k-1}X\quad\text{ for all } i,j\in\set{1,2}. 
\]
\end{enumerate}

\end{lemma}

\begin{proof}
We may  apply Corollary \ref{exceptional-set} with $H=2h_i(i=1,2)$ and $\delta=\eta^2$ respectively to conclude that there is a set $\mathcal E_1\subseteq[X,2X]$ with $\#\mathcal E_1\ll\eta^2 X$ and when $z_{j,x}^{(i)}-h_i  \not\in\mathcal E_1$ for $i,j\in\set{1,2}$ we have
\[
\sum_{|n-z_{j,x}^{(i)}|\leq h_i}\tau_k(n)\ll h_i\log^{k-1}X. 
\]
Thus, if we restrict integers $x\in[X,2X]$ such that the corresponding $z_{j,x}^{(i)}-h_i$  being outside of $\mathcal E_1$ the third conclusion always holds. Besides, the application of Corollary \ref{exceptional-set} with $\delta=\eta^2$ also shows that there is an exceptional set $\mathcal E_2\subseteq[X,2X]$ with $\#\mathcal E_2\ll\eta^2 X$ such that the contribution from $\mathcal E_2$ is negligible in the sense of 
\[
\sum_{x\in\mathcal E_2}\sum_{x<n\leq x+H}(\tau_k(n)+\tau_k^*(n))\ll\eta^2XH\log^{k-1}X,
\]
 and when $x\in[X,2X]\backslash\mathcal E_2$ the first conclusion holds, that is
\[
\sum_{x<n\leq x+H}(\tau_k(n)+\tau_k^*(n))\ll H\log^{k-1}X.
\]
We can further restrict the consideration of the function $\tau_k$ to the set $S$ (defined in (\ref{s-3})) by noting that when $1\leq H\leq X^{1/2-\eps}$, for any $1\leq i\leq 4$,  Lemma \ref{shiu-bound} reveals that
\begin{multline*}
\sum_{x\sim X}\sum_{x<n\leq x+H}1_{(n,\mathcal P(P_i,Q_i))=1}\tau_k(n)
\ll H^2X^\eps+H\frac{X}{\log X}	\exp\bigbrac{\sum_{p\leq P_i}\frac{k}{p}+\sum_{Q_i<p\leq 2X}\frac{k}{p}}\\
\ll \bigbrac{\frac{\log P_i}{\log Q_i}}^kXH\log^{k-1}X\ll\eta^kXH\log^{k-1}X.
\end{multline*}
Meanwhile, we also observe that
\[
\sum_{x\sim X}\sum_{x<n\leq x+H \atop \Omega(n)\geq (1+\eps)k\log\log X} \tau_k(n) \leq \sum_{X<n\leq 3X \atop \Omega(n)\geq (1+\eps)k\log\log X} \tau_k(n) \sum_{n-H\leq x<n}1
\ll H \sum_{X<n\leq 3X \atop \Omega(n)\geq (1+\eps)k\log\log X} \tau_k(n).
\]
Setting $\tilde \eps =\frac{\log \eta^{-k}}{(k\eps /2)\log\log X}$, i.e. $\eta^{-k}=(\log X)^{k\eps\tilde\eps/2}$, as noting $\tilde \eps>0$, one can deduce from  Rankin's trick and    Shiu's bound  (Lemma \ref{shiu-bound})  that
\begin{align*}
\sum_{X<n\leq 3X}\tau_k(n) 1_{\Omega(n)\geq (1+\eps)k\log\log X} &\leq (1+\tilde\eps )^{-(1+\eps)k\log\log X} \sum_{X<n\leq 3X}\tau_k(n)(1+\tilde\eps )^{\Omega(n)}	\\
& \ll X (\log X)^{-(1+\eps )k\log(1+\tilde\eps )} \prod_{p\leq 3X}\bigbrac{1+ \frac{(1+\tilde\eps )k-1}{p}}\\
&\ll X \log^{k-1}X (\log X)^{\tilde\eps k-(1+\eps )k\log(1+\tilde\eps )},
\end{align*}
which is bounded by $O(\eta^k X \log^{k-1}X)$ from the assumption of $\tilde\eps$. Combining the above three estimates one thus has
\[
\sum_{x\sim X}\sum_{x<n\leq x+H} \tau_k(n)  1_{n\not\in S}\ll \eta^k XH\log^{k-1}X
\]
 
On the other hand, it is directly from the assumption that for every $x\in[X,2X]$ there exists a frequency $\vec\alpha_x=(\alpha_x^{(d)},\dots,\alpha_x^{(1)})\in\T^d$ and a number $|x_0-x|\leq H$ such that
\[
\sum_{x\sim X}\Bigabs{ \sum_{x<n\leq x+H}(\tau_k(n)-\tau_k^*(n))e\bigbrac{\sum_{1\leq j\leq d}\alpha_x^{(j)}(n-x_0)^j} } \gg \eta XH\log^{k-1}X.
\]
Therefore, by combining the above two hands together we can conclude that
\begin{multline*}
	\sum_{x\in [X,2X]\backslash\mathcal E_2}\Bigabs{ \sum_{x<n\leq x+H}( 1_S(n) \tau_k(n)-\tau_k^*(n))e\bigbrac{\sum_{1\leq j\leq d}\alpha_x^{(j)}(n-x_0)^j} }\\
	\geq \sum_{x\in(X,2X]} \Bigabs{ \sum_{x<n\leq x+H}(\tau_k(n)-\tau_k^*(n))e\bigbrac{\sum_{1\leq j\leq d}\alpha_x^{(j)}(n-x_0)^j} }\\
	-\sum_{n\in\mathcal E_2}(\tau_k(n)+\tau_k^*(n))-\sum_{x\sim X}\sum_{x<n\leq x+H} \tau_k(n) 1_{n\not\in S}\gg \eta XH\log^{k-1}X.
\end{multline*}

Noting that the pointwise upper bound when $x\not\in\mathcal E_2$, one can deduce from  the pigeonhole principle that there is a set $\mathcal A\subseteq [X,2X]\backslash\mathcal E_2$ with cardinality $\#\mathcal A\gg\eta X$ such that when $x\in\mathcal A$ we have
\[
\Bigabs{ \sum_{x<n\leq x+H}(1_S(n)\tau_k(n)-\tau_k^*(n))e\bigbrac{\sum_{1\leq j\leq d}\alpha_x^{(j)}(n-x_0)^j} }\gg\eta H\log^{k-1}X.
\]
One may remove elements in $\mathcal A$ such that the corresponding $z_{j,x}^{(i)}-h_i$  ($i,j\in\set{1,2})$  being in the set $\mathcal E_1$ and rename the set of the remaining elements to  $\mathcal A$, then we still have $\#\mathcal A\gg\eta X$.
Now we may decompose the elements of $\mathcal A$ into $H$-separated subsets, for example, assume that $0\leq i<H$ and denote $\mathcal A_i=\set{x\in \mathcal A: x\equiv i\mod H}$. The pigeonhole principle implies that there is some $\mathcal A_i$ satisfies $\#\mathcal A_i\gg\eta X/H$. Taking  $\mathcal J_{-1}=\set{(x,\vec\alpha_x):x\in\mathcal A_i}$, then $\mathcal J_{-1}\subseteq[X,2X]\times\T^d$, and according to Definition \ref{configuration}, $\mathcal J_{-1}$ is a $(c\eta,H)$-configuration. 
\end{proof}

\begin{lemma}[Extracting a small prime factor] \label{4-1}
Suppose that $\exp(\log^\eps X)\leq H\leq X^{1/2-\eps}$ and
\[
\int_X^{2X}\sup_{\vec\alpha\in\T^d} \sup_{|x_0-x|\leq H}\Bigabs{\sum_{x<n\leq x+H} (\tau_k(n)-\tau^*_k(n)) e\bigbrac{\sum_{1\leq j\leq d}\alpha^{(j)}(n-x_0)^j}}\,\rd x\geq\eta XH \log^{k-1}X.
\]
Then there is a $(c\eta^2,H)$-configuration $\mathcal J\subseteq[X,2X]\times\T^d$ and a number $P\in[Q_0,Q/2]$ such that the following statement holds.

For each $(x,\vec\alpha_x)\in \mathcal J$ there is some number $x_0$ with $|x_0-x|\leq H$ such that
\[
\Bigabs{\sum_{\frac{x}{p}<n\leq\frac{x+H}{p}}\frac{1_S(n)\tau_k(n)-t_n\tau^*_k(n)}{\pq(n)+1} e\bigbrac{ \sum_{1\leq j\leq d}p^j\alpha_x^{(j)} (n-x_0/p)^j}}\gg\eta^4\frac{H\log^{k-1}X}{p}
\]
 holds for $\gg\eta^{6}P/\log P$  primes $p\in[P,2P]$, here, $0<t_n<1$ for each integer $n\in[X,2X]$ and $S\subseteq[X,2X]$ is the set defined in (\ref{s-3}).

\end{lemma}

\begin{proof}
Building on the second conclusion of Lemma \ref{discretization}  there is a $(c\eta,H)$-configuration $\mathcal J_{-1}\subseteq[X,2X]\times\T^d$ for which when $(x,\vec\alpha_x)\in\mathcal J_{-1}$ there is some $x_0$ with $|x_0-x|\leq H$ such that
\[
\bigabs{\sum_{x<n\leq x+H}(1_S(n)\tau_k(n)-\tau_k^*(n))e\bigbrac{\sum_{1\leq j\leq d}\alpha_x^{(j)}(n-x_0)^j}}\gg\eta H\log^{k-1}X.
\]
Recalling Definition \ref{configuration}, we may sum over $(x,\vec\alpha_x)\in\mathcal J_{-1}$ to conclude that
\[
\eta^2 X\log^{k-1}X\ll\sum_{ (x,\vec\alpha_x)\in\mathcal J_{-1} } \Bigabs{ \sum_{x<n\leq x+H}(1_S(n)\tau_k(n)-\tau_k^*(n)) e\bigbrac{\sum_{1\leq j\leq d}\alpha_x^{(j)}(n-x_0)^j} }.
\]
Based on whether or not $n$ has a prime factor in the interval $(Q_0, Q]$, the right-hand side summation is directly bounded by the sum of the following two expressions:
\[
\sum_{ (x,\vec\alpha_x)\in\mathcal J_{-1} } \Bigabs{ \sum_{x<n\leq x+H} 1_{(n,\mathcal P(Q_0,Q))>1}(1_S(n)\tau_k(n)-\tau^*_k(n)) e\bigbrac{\sum_{1\leq j\leq d}\alpha_x^{(j)}(n-x_0)^j}}
\]
and 
\[
\twosum{x\sim X}{x:H\text{-separated}} \sum_{x<n\leq x+H}\bigbrac{1_{(n,\mathcal P(Q_0,Q))=1} \tau_k(n)+1_{(n,\mathcal P(Q_0,Q))=1}\tau^*_k(n)} \ll\sum_{n\sim X}1_{(n,\mathcal P(Q_0,Q))=1} \tau_k(n),
\]
by noting that $\tau_k^*\ll\tau_k$ pointwise. We are going to show that the second expression is negligible based on our chosen parameters $Q_0$ and $Q$. In practice, Lemma \ref{shiu-bound} implies that
\[
\sum_{n\sim X}1_{(n,\mathcal P(Q_0,Q))=1} \tau_k(n)\ll\frac{X}{\log X} \exp\Bigbrac{\sum_{p\leq Q_0}\frac{k}{p}+\sum_{Q\leq p\leq 2X}\frac{k}{p}}\ll X\log^{k-1}X\bigbrac{\frac{\log Q_0}{\log Q}}^k.
\] 
Recalling that $Q_0=Q^\omega$, this term is thereby bounded by $O(\eta^3X\log^{k-1}X)$ as $\omega$ is sufficiently small concerning $\eta$, which is negligible. 

Therefore, we may assume that
\[
\eta^2 X\log^{k-1}X\ll \sum_{ (x,\vec\alpha_x)\in\mathcal J_{-1} } \Bigabs{ \sum_{x<n\leq x+H} 1_{(n,\mathcal P(Q_0,Q))>1}(1_S(n)\tau_k(n)-\tau^*_k(n)) e\bigbrac{\sum_{1\leq j\leq d}\alpha_x^{(j)}(n-x_0)^j}}.
\]
The application of Ramar\'e's identity $1_{(n,\mathcal P(Q_0,Q))>1}=\sum_{Q_0<p\leq Q}\sum_{pm=n}\frac{1}{\pq(pm)}$ yields that
\[
\eta^2 X\log^{k-1}X\ll \sum_{ (x,\vec\alpha_x)\in\mathcal J_{-1} } \sum_{Q_0<p\leq Q}\Bigabs{\sum_{\frac{x}{p}<n\leq\frac{x+H}{p}}\frac{1_S(pn)\tau_k(pn)-\tau^*_k(pn)}{\pq(pn)} e\bigbrac{\sum_{1\leq j\leq d}p^j\alpha_x^{(j)}(n-\frac{x_0}{p})^j}}.
\]
Noting that the contribution of integers $n$ in the range $(\frac{x}{p},\frac{x+H}{p}]$ which are not coprime with $p$ can be estimated by
\[
\twosum{x\sim X}{x:H\text{-separated}} \sum_{Q_0<p\leq Q}\sum_{x<p^2m\leq x+H}\tau_k(p^2n)\ll X\log^{3k}X\sum_{Q_0<p\leq Q}p^{-2}\ll Q_0^{-1}X\log^{3k}X,
\]
which is negligible.
Taking note that $1_S(pn)=1_S(n)$ when $Q_0<p\leq Q$, along with $\pq(pn)=\pq(n)+1$, $\tau_k(pn)=k\tau_k(n)$ and 
\[\tau_k^*(pn)=\gamma^{1-k}\tau_{k-1}(p)\sum_{m|n \atop m\leq (X/p)^\gamma}\tau_{k-1}(m)+\gamma^{1-k}\sum_{m|n\atop m\leq X^\gamma}\tau_{k-1}(m)
\]
 when $(p,n)=1$ and $p\leq Q< X^\gamma$, therefore, by assuming that $0<t_n<1$ is the number such that $\tau_k^*(pn)=kt_n\tau_k^*(n)$, we can conclude that
\[
\eta^2 X\log^{k-1}X\ll\sum_{ (x,\vec\alpha_x)\in\mathcal J_{-1} }  \sum_{Q_0<p\leq Q} \Bigabs{\sum_{\frac{x}{p}<n\leq\frac{x+H}{p}}\frac{1_S(n)\tau_k(n)-t_n\tau^*_k(n)}{\pq(n)+1}e\bigbrac{\sum_{1\leq j\leq d}p^j\alpha_x^{(j)}(n-\frac{x_0}{p})^j} }.
\]
On the other hand, taking note that when $(x,\vec\alpha_x) \in\mathcal J_{-1}$, it can be deduced from Ramar\'e's identity and the first conclusion of Lemma \ref{discretization} that 
\begin{multline*}
\sum_{Q_0<p\leq Q}\sum_{\frac{x}{p}<n\leq\frac{x+H}{p}}\frac{\tau_k(n)+\tau^*_k(n)}{\pq(n)+1}\ll\sum_{x<m\leq x+H}\sum_{Q_0<p\leq Q}\sum_{m=pn}\frac{\tau_k(m)+\tau^*_k(m)}{\pq(n)+1}\\
\ll	\sum_{x<m\leq x+H}(\tau_k(m)+\tau^*_k(m))\ll H\log^{k-1}X.
\end{multline*}
Combining the above two estimates together yields a set $\mathcal J\subset\mathcal J_{-1}$ with $\#\mathcal J\gg\eta^2X/H$ such that when $(x,\vec\alpha_x)\in\mathcal J$ we have
\[
\eta^2 H\log^{k-1}X\ll\sum_{Q_0<p\leq Q} \Bigabs{\sum_{\frac{x}{p}<n\leq\frac{x+H}{p}}\frac{1_S(n)\tau_k(n)-t_n\tau^*_k(n)}{\pq(n)+1}e\bigbrac{\sum_{1\leq j\leq d}p^j\alpha_x^{(j)}(n-\frac{x_0}{p})^j} }.
\]
Besides, it is obvious from Definition \ref{configuration} that this $\mathcal J\subseteq[X,2X] \times \T^d$ is a $(c\eta^2, H)$-configuration.

We now claim that there is a subset of primes $\mathcal P\subseteq(Q_0,Q]$ satisfying $\sum_{p\in\mathcal P}\frac{1}{p}\gg\eta^6\sum_{Q_0<p\leq Q}\frac{1}{p}$ and  such that for each $p\in\mathcal P$ the following estimate holds uniformly for  $(x,\vec\alpha_x)\in\mathcal J$ 
\begin{align}\label{bias-p}
\Bigabs{\sum_{\frac{x}{p}<n\leq\frac{x+H}{p}}\frac{1_S(n)\tau_k(n)-t_n\tau^*_k(n)}{\pq(n)+1} e\bigbrac{\sum_{1\leq j\leq d}p^j\alpha_x^{(j)}(n-\frac{x_0}{p})^j} }\gg\eta^4 \frac{H\log^{k-1}X}{p}.
\end{align}
It then follows from the pigeonhole principle that, if we decompose the interval $(Q_0,Q]$ into dyadic intervals $(P_j,2P_j]$ with $P_j=2^jQ_0$, there must be some $P_j$ such that $\sum_{p\in\mathcal P\cap(P_i,2P_i]}\frac{1}{p}\gg\eta^6\sum_{p\sim P_i}\frac{1}{p}$. We can thus deduce from prime number theorem that there are $\gg\eta^6\frac{P_i}{\log P_i}$  primes $p\in (P_i,2P_i]$ such that (\ref{bias-p}) holds. The lemma then follows by taking $P$ as this $P_j$. 

It remains to prove the claim (\ref{bias-p}). It is easy to see that when  $(x,\vec\alpha_x)\in\mathcal J$ one has
\begin{multline*}
 \sum_{p\in\mathcal P}\Bigabs{\sum_{\frac{x}{p}<n\leq\frac{x+H}{p}}\frac{1_S(n)\tau_k(n)-t_n\tau^*_k(n)}{\pq(n)+1} e\bigbrac{\sum_{1\leq j\leq d}p^j\alpha_x^{(j)}(n-\frac{x_0}{p})^j} }+\sum_{Q_0<p\leq Q}\eta^4\frac{H\log^{k-1}X}{p}\\
 \gg	\eta^2 H\log^{k-1}X.
\end{multline*}
Given our assumptions on the parameters $Q_0$ and $Q$, Mertens' theorem implies that
\[
\sum_{Q_0<p\leq Q}\frac{1}{p}\ll\log\log Q-\log\log Q_0\ll\log\log H^\omega-\log\log H^{\omega^2}\ll\eta^{-2},
\]
if $\omega$ is sufficiently small with respect to $\eta$. This implies the second summation term has a negligible contribution. By summing over $(x,\vec\alpha_x)\in\mathcal J$ one thus has
\[
\eta^4 X\log^{k-1}X\ll \sum_{p\in\mathcal P} \sum_{(x,\vec\alpha_x)\in\mathcal J}\Bigabs{\sum_{\frac{x}{p}<n\leq\frac{x+H}{p}}\frac{1_S(n)\tau_k(n)-t_n\tau^*_k(n)}{\pq(n)+1}e\bigbrac{\sum_{1\leq j\leq d}p^j\alpha_x^{(j)}(n-\frac{x_0}{p})^j} }.
\]
By combining the above estimate with the upper bound below
\[
\twosum{x\sim X}{x:H\text{-separated}}\sum_{\frac{x}{p}<n\leq\frac{x+H}{p}}(\tau_k(n)+t_n\tau^*_k(n))\ll\frac{X\log^{k-1}X}{p},
\]
one has $\sum_{p\in\mathcal P}\frac{1}{p}\gg\eta^{4}\gg \eta^{4}\frac{\sum_{Q_0<p\leq Q}\frac{1}{p}}{\eta^{-2}}$. 
\end{proof}


\section{The configuration is somewhat fully  connected}

In this section, we aim to show that, to some extent, the configuration we constructed in Lemma \ref{discretization} is fully connected. Specifically, we will define a type of relationship between any two elements in a configuration and ultimately identify numerous pairs of elements in $\mathcal J$ that are related. Before delving into this, we require some auxiliary results concerning exponential sum estimates.

\begin{lemma}[Only a few frequencies are inherently dominant]\label{exponential-sum}
Suppose that $\log^{10000k\log k}X\leq H\leq X^{1/2-\eps}$ and $x,z\in[X,2X]$ with $|z-x|\leq H$. Suppose that there is a set $\mathcal E\subseteq[-H,H]$ with $\#\mathcal E\ll\eta^6H$ satisfies 
\[
\sum_{h\in\mathcal E}\sum_{|n-x|\leq H}\tau_k(n)\tau_k(n+h)\ll\eta^6 H^2\log^{2k-2}X,
\] 
and  $\sum_{|n-x|\leq H}\tau_k(n)\tau_k(n+h)\ll\eta^{-3}H\log^{2k-2}X$ whenever $h\not\in\mathcal E$. Then
 there is  a set $\set{\vec\beta_i}_{1\leq i\leq K'}\subset\T^d$ with cardinality $K'\ll\eta^{-O(1)}$ such that the following holds. If $\vec\alpha\in\T^d$ is a frequency such that
\[
\eta H\log^{k-1}X\leq\Bigabs{\sum_{x<n\leq x+H}\frac{1_S(n)\tau_k(n)-t_n\tau^*_k(n)}{\pq(n)+1}e\bigbrac{\sum_{1\leq j\leq d}\alpha^{(j)}(n-z)^j}},
\]
 where $S\subseteq[X,2X]$ is defined in (\ref{s-3}) and $0<t_n<1$ for each $n\in[X,2X]$, then there is some  $\vec\beta\in\set{\vec\beta_i}_{1\leq i\leq K'}$ and some integer $1\leq q\ll\eta^{-O(1)}$ such that 
\[
\norm{q(\beta^{(j)}-\alpha^{(j)})}\ll\eta^{-O(1)}H^{-j}\qquad(1\leq j\leq d).
\]	
\end{lemma}

\begin{proof}
	
Without loss of generality, we may assume that	$\vec\alpha_1,\dots,\vec\alpha_K$ are all of the frequencies satisfying this inequality and also satisfying $\norm{\alpha_i^{(d)}-\alpha_{i'}^{(d)}}\gg H^{-d}$ for distinct $i\neq i'$. Summing over $\vec\alpha_1,\dots,\vec\alpha_K$  obtains that
\[
\eta HK\log^{k-1}X\leq\sum_{1\leq i\leq K}\Bigabs{\sum_{x<n\leq x+H}\frac{1_S(n)\tau_k(n)-t_n\tau^*_k(n)}{\pq(n)+1}e\bigbrac{\sum_{1\leq j\leq d}\alpha_i^{(j)}(n-z)^j} }.
\]
Square both sides, on noting that $\tau_k^*\ll\tau_k$ pointwise, one can get from Cauchy-Schwarz inequality and changing the order of summation  that
\begin{multline*}
\eta^2 H^2K\log^{2k-2}X	\ll\sum_{x<n\leq x+H}\sum_{x<n+h\leq x+H}\tau_k(n)\tau_k(n+h)\\
\Bigabs{\sum_{1\leq i\leq K}e\bigbrac{\alpha_i^{(d)}((n+h-z)^d-(n-z)^d)+\dots+\alpha_i^{(1)}((n+h-z)-(n-z))}}\\
\ll	\sum_{|h|\leq H}\sum_{x<n\leq x+H}\tau_k(n)\tau_k(n+h)\Bigabs{\sum_{1\leq i\leq K} e\bigbrac{\alpha^{(d)}_ih^d+P(h;\vec\alpha_i,n-z)}},
\end{multline*}
where $P(h;\vec\alpha_i,n-z)$ is a polynomial of degree at most $d-1$ with respect to the variable $h$. Since the assumption asserts that the contribution from $h\in\mathcal E$ can be bounded by $O(\eta^6H^2K\log^{2k-2}X)$ which is negligible, we can conclude that
\begin{multline*}
\eta^2 H^2K\log^{2k-2}X	\ll\sum_{h\in[-H,H]\backslash \mathcal E }\sum_{x<n\leq x+H}\tau_k(n)\tau_k(n+h)\Bigabs{\sum_{1\leq i\leq K} e\bigbrac{\alpha^{(d)}_ih^d+P(h;\vec\alpha_i,n-z)}}\\
\ll\eta^{-3}H\log^{2k-2}X\sum_{h\in[-H,H]\backslash \mathcal E }\Bigabs{\sum_{1\leq i\leq K} e\bigbrac{\alpha^{(d)}_ih^d+P(h;\vec\alpha_i,n-z)}},
\end{multline*}
by the pointwise bound for $\sum_{x<n\leq x+H}\tau_k(n)\tau_k(n+h)$ whenever $h\notin\mathcal E$. Taking note that $\#\mathcal E\ll\eta^6H$, we thus have
\[
\eta^5HK\ll\sum_{|h|\leq H}\Bigabs{\sum_{1\leq i\leq K} e\bigbrac{\alpha^{(d)}_ih^d+P(h;\vec\alpha_i,n-z)}}.
\]
Taking square both sides, one may find from Cauchy-Schwarz inequality that
\[
\eta^{10}K^2H\ll\sum_{1\leq i,i'\leq K}\Bigabs{\sum_{|h|\leq H}e\bigbrac{h^d(\alpha_i^{(d)}-\alpha_{i'}^{(d)})+P(h;\vec\alpha_i,\vec\alpha_{i'},n-z)}},
\]
where $P(h;\vec\alpha_i,\vec\alpha_{i'},n-z)$ is a polynomial of degree at most $d-1$ with respect to $h$. Pigeonhole principle then shows that there are $\gg\eta^{10}K^2$ of pairs $(i,i')\in[1,K]^2$ such that
\[
\eta^{10}H\ll\Bigabs{\sum_{|h|\leq H}e\bigbrac{h^d(\alpha_i^{(d)}-\alpha_{i'}^{(d)})+P(h;\vec\alpha_i,\vec\alpha_{i'},n-z)}}.
\]
Unless $K\ll\eta^{-10}$ --- in which case we can choose $\set{\vec\gamma_{d,j}}_{1\leq j\leq K'}$ as the initial sequence $\set{\vec\alpha_i}_{1\leq i\leq K}$, there must be a pair $(\vec\alpha_i,\vec\alpha_{i'})$ with $i\neq i'$ such that the above inequality holds. By Weyl's bound \cite[Lemma 1.1.16]{Tao} there is some integer $1\leq q'\ll\eta^{-O(1)}$ such that $\norm{q'(\alpha_i^{(d)}-\alpha_{i'}^{(d)})}\ll\eta^{-O(1)}H^{-d}$. Therefore, in this case, we may rule either $\vec\alpha_i$ or $\vec\alpha_{i'}$ out of the set $\set{\vec\alpha_i}_{1\leq i\leq K}$. It is possible to continue the excluding process until there are only $O(\eta^{-10})$ elements remaining. And by taking $\set{\vec\gamma_{d,j}}_{1\leq j\leq K'}$ as the remaining frequencies,  the desired conclusion holds for degree $j=d$. 

We claim that the desired set for which the entire conclusion holds can be found by induction of the argument on degrees $1\leq j\leq d$. Let $1\leq j'\leq d-1$ and $\set{\vec\gamma_{j'+1,i}}_{1\leq i\ll\eta^{-O(1)}}\subset\T^d$ be a set of frequencies. Suppose that $\vec\alpha\in\T^d$ is a frequency such that the following two conditions hold:
\begin{enumerate}
	\item 
	\begin{align}\label{large-exp}
	\Bigabs{\sum_{x<n\leq x+H}\frac{1_S(n)\tau_k(n)-t_n\tau_k^*(n)}{\pq(n)+1}e\bigbrac{\sum_{1\leq j\leq d}\alpha^{(j)}(n-z)^j} }\geq\eta H\log^{k-1}X	
	\end{align}
\item there is an integer $q'\ll\eta^{-O(1)}$, a large constant $A>1$ and a frequency $\vec\gamma\in\set{\vec\gamma_{j'+1,i}}_{1\leq i\ll\eta^{-O(1)}}$ such that
\begin{align}\label{gamma-alpha}
\norm{q'(\gamma^{(j)}-\alpha^{(j)})}\ll\eta^{-A}H^{-j} \text{ for all } j'<j\leq d.	
\end{align}

\end{enumerate}
Our goal is to find another set of frequencies $\set{\vec \gamma_{j',i}}_{1\leq i\ll \eta^{-O(1)}}$ such that (\ref{gamma-alpha})  holds for $j'\leq j\leq d$. To do this, let us fix an integer $q'\ll\delta^{-O(1)}$ and a frequency $\vec\gamma\in\set{\vec\gamma_{j'+1,i}}_{1\leq i\ll\eta^{-O(1)}}$. Suppose that $\vec\alpha_1,\dots,\vec\alpha_K$ are all of the frequencies such that (\ref{large-exp}) and (\ref{gamma-alpha}) hold with this particular $\vec\gamma$ and $q'$, besides, we also assume that $\norm{\alpha_i^{(j')}-\alpha_{i'}^{(j')}}\gg H^{-j'}$ whenever $i\neq i'$.

Decomposing the interval $[x,x+H]$ into sub-intervals of length $\eta^{2A}H$, on following the argument of degree $j=d$, one has
\begin{multline*}
	\eta HK\log^{k-1}X\leq\sum_{x_0}\sum_{1\leq i\leq K}\Bigabs{\sum_{x_0<n\leq x_0+\eta^{2A}H}\frac{1_S(n)\tau_k(n)-t_n\tau_k^*(n)}{\pq(n)+1}e\bigbrac{\sum_{1\leq j\leq d}\alpha_i^{(j)}(n-z)^j}}\\
	\leq \sum_{x_0}\sum_{1\leq i\leq K}\Bigabs{\sum_{x_0<n\leq x_0+\eta^{2A}H}\frac{1_S(n)\tau_k(n)-t_n\tau_k^*(n)}{\pq(n) +1}\prod_{1\leq j\leq d} e\bigbrac{\gamma^{(j)}(n-z)^j) e((\alpha_i^{(j)}-\gamma^{(j)})(n-z)^j}},
\end{multline*}
where $x_0$ takes values of $\eta^{2A}H$-separated points in the interval $(x,x+H]$ so that $(x,x+H]=\bigsqcup_{x_0}(x_0,x_0+\eta^{2A}H]$. We may decompose $(x_0,x_0+\eta^{2A}H]$ into arithmetic progressions according to the residue classes $\mod{q'}$ to get that
\begin{multline*}
\sum_{r\mod{q'}}\sum_{x_0}\sum_{1\leq i\leq K}\Bigabs{\sum_{x_0<n\leq x_0+\eta^{2A}H\atop n\equiv r\mod{q'}}\frac{1_S(n)\tau_k(n)-t_n\tau_k^*(n)}{\pq(n)+1}\\
\prod_{1\leq j\leq d} e\bigbrac{\gamma^{(j)}(n-z)^j) e((\alpha_i^{(j)}-\gamma^{(j)})(n-z)^j}}
\gg\eta	HK\log^{k-1}X.
\end{multline*}
Noting that  our assumption (\ref{gamma-alpha}) implies that when $j'<j\leq d$ there are integers $1\leq a_j\leq q'$ such that
\[
\alpha_i^{(j)}-\gamma^{(j)}=\frac{a_j}{q'}+O(\eta^{-A}H^{-j}),
\]
one may conclude from  Taylor expansion that whenever $1\leq n-x_0\leq \eta^{2A}H$, $|x_0-z|\ll H$ and $j'<j\leq d$ that  the phase functrion $e\bigbrac{(n-z)^j(\alpha_i^{(j)}-\gamma^{(j)})}$ is equal to 
\begin{multline*}
e\bigbrac{(x_0-z)^j(\alpha_i^{(j)}-\gamma^{(j)})}e\Bigbrac{\frac{a_j}{q'}\bigbrac{(n-z)^j-(x_0-z)^j}}e\Bigbrac{O\bigbrac{\eta^{-A}H^{-j}\sum_{1\leq i\leq j}(n-x_0)^i(x_0-z)^{j-i}}}\\
=e\bigbrac{(x_0-z)^j(\alpha_i^{(j)}-\gamma^{(j)})}	e\Bigbrac{\frac{a_j}{q'}\bigbrac{(n-z)^j-(x_0-z)^j}} (1+O(\eta^A)).
\end{multline*}
Taking note that the first factor is independent of $n$ and the second factor is also independent of $n$ if we restrict $n$ into the arithmetic progression $r\mod{q'}$, therefore, 
\begin{multline*}
\eta HK\log^{k-1}X\ll\sum_{r\mod{q'}}\sum_{x_0}\sum_{1\leq i\leq K}\Bigabs{\sum_{x_0<n\leq x_0+\eta^{2A}H\atop n\equiv r\mod{q'}}\\\frac{1_S(n)\tau_k(n)-t_n\tau_k^*(n)}{\pq(n)+1}e\bigbrac{\sum_{1\leq j\leq d}\gamma^{(j)}(n-z)^j}e\bigbrac{(\alpha_i^{(j')}-\gamma^{(j')})(n-z)^{j'}+\cdots+(\alpha_i^{(1)}-\gamma^{(1)})(n-z)}}.
\end{multline*}

Performing as before, one can find that
\begin{multline*}
\eta^2 H^2K\log^{2k-2}X\ll\\
\sum_{|h|\leq H}\sum_{x<n\leq x+H}\tau_k(n)\tau_k(n+h)
\Bigabs{\sum_{1\leq i\leq K}e\bigbrac{(\alpha_i^{(j')}-\gamma^{(j')})h^{j'}+P(h)}},	
\end{multline*}
where $P(h)=P(h;\alpha_i^{(j')}-\gamma^{(j')},\cdots,\alpha_i^{(1)}-\gamma^{(1)},n-z)$ is a polynomial of degree at most $j'-1$ with respect to $h$. Therefore, the previous calculation leads us to
\[
\eta^{10}K^2H\ll\sum_{1\leq i,i'\leq K}\bigabs{\sum_{h\leq H}e\bigbrac{(\alpha_i^{(j')}-\alpha_{i'}^{(j')})h^{j'}+\tilde P(h)}},
\]
where $\tilde P(h)$ is a polynomial of degree at most $j'-1$. Weyl's bound \cite[Lemma 1.1.16]{Tao} then yields that there are at most $O(\eta^{-10})$  members in $\set{\vec\alpha_i}_{1\leq i\leq K}$ forming a sequence $\set{\vec\gamma_{i'}}_{1\leq i'\ll\eta^{-10}}$ such that for every $\vec\alpha_{i}$ there is some $\vec\gamma_{i'}$ and some integer $q^*\ll\eta^{-O(1)}$ such that $\norm{q^*(\gamma_{i'}^{(j')}-\alpha_i^{(j')})}\ll\eta^{-O(1)}H^{-j'}$. Set $q_{j'}=q'q^*$, then
\[
\norm{q_{j'}(\gamma_{i'}^{(j')}-\alpha_i^{(j')})}\ll\eta^{-O(1)}H^{-j'};
\]
besides, it follows from (\ref{gamma-alpha}) and the triangle inequality that whenever $j'< j\leq d$
\[
\norm{q_{j'}(\gamma_{i'}^{(j)}-\alpha_i^{(j)})}\leq \norm{q_{j'}(\gamma_{i'}^{(j)}-\gamma^{(j)})}+\norm{q_{j'}(\alpha_{i}^{(j)}-\gamma^{(j)})}\ll\eta^{-O(1)}H^{-j}.
\]
The claimed sequence $\set{\vec\gamma_{j',i}}_i$ is the union of sequences $\set{\vec\gamma_{i'}}_{1\leq i'\ll\eta^{-10}}$ obtained by varying $\vec\gamma$  over $\set{\vec\gamma_{j'+1,i}}_{i\ll\eta^{-O(1)}}$ and varying $q'$ over $[1,O(\eta^{-O(1)})]$. It is easy to see that $\#\set{\vec\gamma_{j',i}}\ll\eta^{-O(1)}$. And the lemma follows by taking $\set{\vec\beta_i}$ as $\set{\vec\gamma_{1,i}}_{1\leq i\ll\eta^{-O(1)}}$.

\end{proof}

With the assistance of Lemma \ref{exponential-sum}, we can now build a ``narrow" configuration $\mathcal J_1$ that is contained in a narrow interval for which elements in $\mathcal J$ can be related to elements in $\mathcal J_1$. We call this step as \textit{scaling down}.

\begin{proposition}[Scaling down]\label{scale-down}
Suppose that $\log^{10000k\log k}X\leq H\leq X^{1/3-\eps}$, $Q_0\leq P<Q/2$, $\mathcal J\subseteq[X,2X]\times\T^d$ is a $(c\eta^2,H)$-configuration, and for each $(x,\vec\alpha_x)\in\mathcal J$ there exist   $\gg\eta^6 P/\log P$  primes $p\in[P,2P]$ and some $|x_0-x|\leq H$ such that
\[
\Bigabs{ \sum_{\frac{x}{p}<n\leq\frac{x+H}{p}}\frac{1_S(n)\tau_k(n)-t_n\tau^*_k(n)}{\pq(n)+1}e\bigbrac{\sum_{1\leq j\leq d}p^j\alpha_x^{(j)}(n-x_0/p)^j} }\geq \eta^4\frac{H\log^{k-1}X}{p},
\]
where $S\subseteq[X,2X]$ is defined in (\ref{s-3}) and $0<t_n<1$ for every $n\in[X,2X]$.

Then there is a $(c\eta^{O(1)},\frac{H}{P})$-configuration $\mathcal J_1\subseteq[\frac{X}{2P},\frac{2X}{P}]\times\T^d$ for which when $(y,\vec\alpha_y)\in\mathcal J_1$ there exists an integer $1\leq q_y\ll\eta^{-O(1)}$ satisfying the following two events:	
\begin{enumerate}
\item 	there is  some $y_0$ with $|y_0-y|\leq H/P$ such that
\[
\Bigabs{ \sum_{y<n\leq y+\frac{H}{P}}\frac{1_S(n)\tau_k(n)-t_n\tau^*_k(n)}{\pq(n)+1} e\bigbrac{\sum_{1\leq j\leq d}q_y^{-1}\alpha_y^{(j)}(n-y_0)^j} }\gg \eta^{4}\frac{H\log^{k-1}X}{P};
\] 
\item there are $\gg\eta^{O(1)}\frac{P}{\log P}$ elements $(p,(x,\vec\alpha_x))\in[P,2P]\times\mathcal J$  such that
\[
|py-x|\leq 2H,\qquad \alpha_y^{(j)}\equiv p^jq_y\alpha_x^{(j)}+O(\eta^{-O(1)}(P/H)^j)\mod 1\qquad(1\leq j\leq d). 
\]
\end{enumerate}

\end{proposition}

\begin{proof}
	
Corollary \ref{exceptional-set} yields that outside of $\#\mathcal E\ll \eta^{20} \frac{X}{P}$ numbers $s\in[\frac{X}{2P},\frac{2X}{P}]$  the pointwise upper bound below holds
\begin{align}\label{upper-bound}
\sum_{s<n\leq s+\eta^5\frac{H}{P}}(\tau_k(n)+\tau^*_k(n))\ll\eta^5\frac{H\log^{k-1}X}{P}.
\end{align}
 Now for a fixed  pair $(p,(x,\vec\alpha_x)) \in [P,2P]\times \mathcal J$ satisfies
\begin{align}\label{ass-bound}
\Bigabs{ \sum_{\frac{x}{p}<n\leq\frac{x+H}{p}}\frac{1_S(n)\tau_k(n)-t_n\tau^*_k(n)}{\pq(n)+1} e\bigbrac{\sum_{1\leq j\leq d}p^j\alpha_x^{(j)}(n-x_0/p)^j} }\geq \eta^4\frac{H\log^{k-1}X}{p},
\end{align}
it is easy to find that if $I\subseteq[\frac{X}{2P},\frac{2X}{P}]$ is an interval with $|I|\gg \frac{H}{P}$ and $|I\bigtriangleup[\frac{x}{p},\frac{x+H}{p}]|\leq \eta^5 \frac{H}{P}$, by naming the starting point of this interval as $y$, we have $|py-x|\leq 2H$ and $y\in[\frac{X}{2P},\frac{2X}{P}]$, and there are $\gg\eta^5\frac{H}{P}$ such $y$ created by this pair $(p,(x,\vec\alpha_x))$. When $(p,(x,\vec\alpha_x))$ varies, there are in total $\gg\eta^{13}\frac{H}{P}\frac{P}{\log P}\frac{X}{H}$ such $y$ produced (counting  the occurrences). On the contrary, noting that each $y\in[\frac{X}{2P},\frac{2X}{P}]$ can only relate to at most $O(P/\log P)$ of $(p,(x,\vec\alpha_x))$, we can see these $y$ for which either $y-\eta^5\frac{H}{P}$  or  $y+\frac{H}{P}$ is in the exceptional set $\mathcal E$ are sparse. Moreover, the application of Lemma \ref{divisor-correlation} with parameters $\delta=\eta^{15}$, $X=X/H$, and $H=H/P$  implies that the set of $y$ failing to satisfy the conclusion of Lemma \ref{divisor-correlation} is bounded by $O(\eta^{15}\frac{P}{\log P}\frac{X}{P})$ which is also sparse. Thus, after removing those exceptional $y$, we can assume that there are $\gg\eta^{13}\frac{H}{P}\frac{P}{\log P}\frac{X}{H}$ of $(y,p,(x,\vec\alpha_x))$ such that
\begin{enumerate}
	\item $|py-x|\leq 2H$;
	\item (\ref{upper-bound}) holds for both $s=y-\eta^5\frac{H}{P}$ and $s=y+\frac{H}{P}$;
	\item there is an exceptional set $\mathcal E'\subseteq[-H/P,H/P]$ with $|\mathcal E'|\ll\eta^{30}H/P$ and such that when $h\not\in\mathcal E'$ we have $\sum_{|n-y|\leq H/P}\tau_k(n)\tau_k(n+h)\ll\eta^{-15}(H/P)\log^{2k-2}X$, and $\sum_{h\in\mathcal E'}\sum_{|n-y|\leq H/P}\tau_k(n)\tau_k(n+h)\ll\eta^{30}(H/P)^2\log^{2k-2}X$.
\end{enumerate}
 Pigeonhole principle illustrates that there are $\gg\eta^{13}\frac{P}{\log P}\frac{X}{H}$ of such triples $(y,p,(x,\vec\alpha_x))$ for which $y$ are $\frac{H}{P}$-separated. Another application of pigeonhole principle implies that there are $\gg \eta^7\frac{X}{H}$ of $\frac{H}{P}$-separated $y$ such that   the above (1--3) holds, besides, for each of those $y$ there are $\gg\eta^6P/\log P$ of $(p,(x,\vec\alpha_x))$ such that $|py-x|\leq 2H$. We can collect these $y$ into a set $\mathcal Y$.

Let $y\in\mathcal Y$, we may suppose that $I_y$ is the interval such that $\bigabs{[y,y+|I_y|]\bigtriangleup[\frac{x}{p},\frac{x+H}{p}]} \leq \eta^5 \frac{H}{P}$ and  rescale the length of the interval $I_y$ as $\frac{H}{P}$, it then follows from (\ref{ass-bound}) that
\begin{multline*}
\eta^4\frac{H\log^{k-1}X}{p}\leq \Bigabs{\sum_{y<n\leq y+\frac{H}{P}}\frac{1_S(n)\tau_k(n)-t_n\tau^*_k(n)}{\pq(n)+1}e\bigbrac{\sum_{1\leq j\leq d}p^j\alpha_x^{(j)}(n-x_0/p)^j}}\\
	+\Bigabs{\sum_{\frac{x}{p}<n\leq y}(\tau_k(n)+\tau^*_k(n))}+\Bigabs{\sum_{y+\frac{H}{P}<n\leq \frac{x+H}{p}}(\tau_k(n)+\tau^*_k(n))}.
	\end{multline*}
Since our assumption $\bigabs{[y,y+|I_y|]\bigtriangleup[\frac{x}{p},\frac{x+H}{p}]} \leq \eta^5 \frac{H}{P}$ reveals that $|y-\frac{x}{p}|\leq\eta^5\frac{H}{P}$ and $|\frac{x+H}{p}-y-\frac{H}{P}|\leq\eta^5\frac{H}{P}$, taking (\ref{upper-bound}) into account, we can conclude that the last two sums on the right-hand side are negligible. Therefore, when $y\in\mathcal Y$ there is some $|y_0-y|\leq H/P$ such that
\[
\eta^4\frac{H\log^{k-1}X}{P}\ll \Bigabs{\sum_{y<n\leq y+\frac{H}{P}}\frac{1_S(n)\tau_k(n)-t_n\tau^*_k(n)}{\pq(n)+1} e\bigbrac{\sum_{1\leq j\leq d}p^j\alpha_x^{(j)}(n-y_0)^j}}	
\]
holds for $\gg\eta^6P/\log P$ of pairs  $(p,(x,\vec\alpha_x))$ and $|py-x|\leq 2H$.

Fixing such a $y\in\mathcal Y$ for the moment, it is immediately from Lemma \ref{exponential-sum} that there are at most $O(\eta^{-O(1)})$ frequencies $\vec\beta_1,\dots,\vec\beta_K$ such that for each $(p^d\alpha^{(d)}_{x},\cdots,p\alpha^{(1)}_{x})$ there is some $\vec\beta_i\in\set{\vec\beta_1,\dots,\vec\beta_K}$ and some integer $1\leq q\ll\eta^{-O(1)}$ such that 
\[
\norm{q(\beta_i^{(j)}-p^j\alpha^{(j)}_{x})}\ll\eta^{-O(1)}(P/H)^j \qquad(1\leq j\leq d).
\]
 On noting that there are totally $O(\eta^{-O(1)})$ choices of $(\vec\beta_i,q)$, pigeonhole principle tells us that there is some $\vec\beta_i$ and some integer $q$ such that the above inequality holds for $\gg\eta^{O(1)}P/\log P$ pairs $(p,(x,\vec\alpha_x))$. We can set $\vec\alpha_y$ as this $q\vec{\beta_i}$ and set
$\mathcal J_1=\cup_{y\in \mathcal Y}\set{(y,\vec\alpha_y)}$, then $\mathcal J_1\subset[\frac{X}{2P},\frac{2X}{P}]\times\T^d$ is a $(c\eta^{O(1)},H)$-configuration. Taking $z_{1,x}^{(1)}=y$, $h_1=H/P$ in Lemma \ref{discretization} (3) and enlarging $q$ if necessary, it follow from Taylor expansion and the above two inequalities that
\begin{multline*}
\Bigabs{\sum_{y<n\leq y+\frac{H}{P}}\frac{1_S(n)\tau_k(n)-t_n\tau^*_k(n)}{\pq(n)+1}e\bigbrac{\sum_{1\leq j\leq d}q^{-1}\alpha_y^{(j)}(n-y_0)^j}}	\\
= \Bigabs{\sum_{y<n\leq y+H/P}\frac{1_S(n)\tau_k(n)-t_n\tau^*_k(n)}{\pq(n)+1}\prod_{1\leq j\leq d} e(p^j\alpha_x^{(j)}(n-y_0)^j)e\bigbrac{q^{-1}(\alpha_y^{(j)}-qp^j\alpha_x^{(j)})(n-y_0)^j)} }\\
\gg \Bigabs{\sum_{y<n\leq y+H/P}\frac{1_S(n)\tau_k(n)-t_n\tau^*_k(n)}{\pq(n)+1}e\bigbrac{\sum_{1\leq j\leq d}p^j\alpha_x^{(j)}(n-y_0)^j}}-O(\eta^{O(1)}\sum_{y<n\leq y+H/P}\tau_k(n))\\
\gg\eta^4\frac{H\log^{k-1}X}{P}.
\end{multline*}
Therefore, $(y,\vec\alpha_y)\in\mathcal J_1$ satisfies the two conditions of the lemma.

\end{proof}

The next conclusion is a generalization of \cite[Proposition 3.2]{MRT20} and \cite[Proposition 3.6]{MRTTZ}, and this is also the main result of this section.

\begin{proposition}[Connectedness]\label{connected}
Suppose that $\exp(\log^\eps X)\leq H\leq X^{1/3-\eps}$ and
\[
\int_X^{2X}\sup_{\vec\alpha\in\T^d} \sup_{|x_0-x|\leq H} \bigabs{\sum_{x<n\leq x+H} (\tau_k(n)-\tau^*_k(n)) e\bigbrac{\sum_{1\leq j\leq d}\alpha^{(j)}(n-x_0)^j}}\,\rd x\geq\eta HX \log^{k-1}X,
\]
then there are two numbers $P,P'$ with  $H^{\omega^2}<P\leq H^\omega$ and $\bigbrac{\frac{H}{P}}^{\omega^3}<P'<\bigbrac{\frac{H}{P}}^{\omega^2}$, two disjoint subsets of primes $\po,\pt\subset[P,2P]$ with $|\po|,|\pt|\gg\eta^{O(1)}\frac{P}{\log P}$, and a $(c\eta^2,H)$-configuration $\mathcal J\subseteq[X,2X]\times\T^d$ such that the following statement holds.

\begin{enumerate}
	\item There are $\gg\eta^{O(1)}(\frac{P}{\log P})^2\frac{X}{H}$ of quadruple 	$(p,q,(x_1,\vec\alpha_{x_1}),(x_2,\vec\alpha_{x_2}))\in\po\times\pt\times\mathcal J^2$ such that $|\frac{x_1}{p}-\frac{x_2}{q}|\ll \frac{H}{P}$ and 
\[
p^j\alpha^{(j)}_{x_1}\equiv q^j\alpha^{(j)}_{x_2}+ O\bigbrac{\eta^{-O(1)}(P/H)^j} \mod {p'} \qquad (1\leq j\leq d)
\]
hold for $\gg\eta^{O(1)}\frac{P'}{\log P'}$  primes $p'\in[P',2P']$;
\item for each $(x,\vec\alpha_x)\in\mathcal J$  there is a number $x_0$ with $|x_0-x|\leq H$  such that
\[
\bigabs{\sum_{x<n\leq x+H}(1_S(n)\tau_k(n)-\tau_k^*(n))e\bigbrac{\sum_{1\leq j\leq d}\alpha_x^{(j)}(n-x_0)^j}}\gg\eta H\log^{k-1}X.
\]
\end{enumerate} 
\end{proposition}

\begin{proof}
We aim to demonstrate that the configuration $\mathcal{J}$ we established in Lemma \ref{4-1} satisfies the two conditons of this proposition. Drawing from the proof of Lemma \ref{4-1}, where it is established that $\mathcal{J}$ is a subset of $\mathcal{J}_{-1}$, the second conclusion immediately follows from Lemma \ref{discretization} (2). Let's now focus on proving the first conclusion. By observing that $[H^{\omega^2}, H^\omega] \cap [(H/P)^{\omega^3}, (H/P)^{\omega^2}] = \emptyset$, one can employ the scaling-down step (Proposition \ref{scale-down}) twice in succession to conclude that---there is some $P'\in[\bigbrac{\frac{H}{P}}^{\omega^3},\bigbrac{\frac{H}{P}}^{\omega^2}]$, a $(c\eta^{O(1)},\frac{H}{PP'})$-configuration $\mathcal J_2\subset[\frac{X}{4PP'},\frac{4X}{PP'}]\times\T^d$ satisfying that for  each $(z,\vec\alpha_z)\in\mathcal J_2$ there is an integer $0<q_z\ll\eta^{-O(1)}$ and $\gg\eta^{O(1)}\frac{P'}{\log P'}$ of $(p',(y,\vec\alpha_y))\in[P',2P']\times\mathcal J_1$ such that
\[
|p'z-y|\ll\frac{H}{P},\qquad\alpha_z^{(j)} \equiv p'^jq_zq_y^{-1}\alpha_y^{(j)}+O(\eta^{-O(1)}(PP'/H)^j)\mod 1\qquad (1\leq j\leq d),
\]
where $1\leq q_y\ll\eta^{-O(1)}$ is the integer in Proposition \ref{scale-down}. The application of the pigeonhole principle reveals that there is a popular modulus $q'\ll\eta^{-O(1)}$ such that $q'=q_z$ for $\gg\eta^{O(1)}|\mathcal J_2|$ of $(z,\vec\alpha_z)$. We can narrow down $\mathcal{J}_2$ to the subset of elements $(z, \vec{\alpha}_z)$ where the corresponding $q_z$ equals this popular $q'$. Fixing the degree $1\leq j\leq d$, we first upgrade the above congruent equation to modulus $q'p'^j$ by adding an integer with an appropriate residue class mudulus $q'p^j$ to $\alpha_z^{(j)}$ and renaming the new frequency as $\alpha_z^{(j)}$  to obtain that
\[
\alpha_z^{(j)} \equiv q'p'^jq_y^{-1}\alpha_y^{(j)}+O(\eta^{-O(1)}(PP'/H)^j)\mod{q'p'^j}.
\]
By the Chinese remainder theorem we can select some $\alpha_z^{(j)}$ such that the above congruent equation holds for $\gg\eta^{O(1)}P'/\log P'$  primes $p'\in[P',2P']$.
 Next, we combine this result with the second conclusion of Proposition \ref{scale-down} to obtain  an integer $1\leq q'\ll\eta^{-O(1)}$ such that the following statement holds. For each $(z,\vec\alpha_z)\in\mathcal J_2$ there are $\gg\eta^{O(1)}\frac{P}{\log P}\frac{P'}{\log P'}$ of triples $(p,p',(x,\vec\alpha_x)) \in [P,2P]\times [P',2P']\times \mathcal J$ such that
\[
|z-\frac{x}{pp'}|\ll\frac{H}{PP'}
\]
and
\[
\alpha_z^{(j)}\equiv q'p'^jp^j\alpha_x^{(j)} + O(\eta^{-O(1)}(PP'/H)^j) \mod {q'p'^j}\qquad (1\leq j\leq d).
\]
Using $(*)$ to denote the above relations hold. It follows from  Cauchy-Schwarz inequality that
\begin{multline*}
	\eta^{O(1)}(\frac{P}{\log P})^2(\frac{P'}{\log P'})^2(\frac{X}{H})^2\ll\Bigbrac{\sum_{(z,\vec\alpha_z)\in\mathcal J_2}\sum_{(x,\vec\alpha_x)\in\mathcal J}\sum_{p\sim P}\sum_{p'\sim P'}1_{(*)}(p,p',(x,\vec\alpha_x),(z,\vec\alpha_z))}^2\\
	\ll\frac{X}{H}\frac{P'}{\log P'}\sum_{(z,\vec\alpha_z)\in\mathcal J_2}\sum_{p'\sim P'}\sum_{p_1\sim P\atop p_2\sim P}\sum_{(x_1,\vec\alpha_{x_1})\in \mathcal J\atop(x_2,\vec\alpha_{x_2})\in \mathcal J} 1_{(*)}(p_1,p',(x_1,\vec\alpha_{x_1}),(z,\vec\alpha_{z}))1_{(*)}(p_2,p',(x_2,\vec\alpha_{x_2}),(z,\vec\alpha_{z})),
\end{multline*}
which means that, there are $\gg\eta^{O(1)}(P/\log P)^2(P'\log P')(X/H)$ of sextuples $(p',p_1,p_2,(z,\vec\alpha_z),(x_1,\vec\alpha_{x_1}),(x_2,\vec\alpha_{x_2}))\in[P',2P']\times[P,2P]^2\times\mathcal J_{2}\times\mathcal J^2$ such that
\[
|z-\frac{x_1}{p_1p'}|\ll H/(PP'), \qquad |z-\frac{x_2}{p_2p'}|\ll H/(PP')
\]
and for every degree $1\leq j\leq d$ we also have
\begin{align*}
&\alpha_{z}^{(j)} \equiv q' p'^jp_1^j\alpha_{x_1}^{(j)}+O(\eta^{-O(1)}(PP'/H)^j)\mod{q'p'^j}\\
& \alpha_{z}^{(j)} \equiv q' p'^jp_2^j\alpha_{x_2}^{(j)}+O(\eta^{-O(1)}(PP'/H)^j)\mod{q'p'^j}.	
\end{align*}
Hence, by combining the above two congruence equations together and subsequently  dividing both sides  by $q'p'^j$, we obtain that
\[
p_1^j\alpha_{x_1}^{(j)}\equiv p_2^j\alpha_{x_2}^{(j)} +O(\eta^{-O(1)}(P/H)^j) \mod 1.
\]
It is possible to upgrade this congruent equation to modulus $p'$ by adding a suitable integer with appropriate residue class mudulus $p'$ to $\alpha_{x_1}^{(j)}$ to get that
\[
p_1^j\alpha_{x_1}^{(j)}\equiv p_2^j\alpha_{x_2}^{(j)} +O(\eta^{-O(1)}(P/H)^j) \mod{p'},
\] 
without affecting the local correlation
\[
\bigabs{\sum_{x_1<n\leq x_1+H}(1_S(n)\tau_k(n)-\tau_k^*(n))e\bigbrac{\sum_{1\leq j\leq d}\alpha_{x_1}^{(j)}(n-x_0)^j}}\gg\eta H\log^{k-1}X.
\]
We then apply Chinese remainder theorem to conclude that this congruent equation holds for $\gg\eta^{O(1)}P'/\log P'$ primes $p'\in[P',2P']$. The conclusion follows by observing that we can apply the pigeonhole principle to restrict $p,q\in[P,2P]$ into two disjoint dense  (up to a $O(\eta^{O(1)})$ factor) subsets $\po$ and $\pt$. 
\end{proof}


\section{Long disjoint pair of paths}

In this section, we are going to use the frequency relations established in Section 2.2 to construct long split paths. In what follows are the recordings of \cite[Definitions 3.6 and 3.8]{Wal23-1}.

\begin{definition}[Split path]\label{path}
Let $Q\in\N$, $\mathcal J^\#\subseteq \mathcal J$	and $\mathcal P_1,\mathcal P_2\subset[P,2P]$ are disjoint subsets. We say $(x_1,\vec\alpha_1)$ and $(x_{k+1}\vec\alpha_{k+1})\in\mathcal J^\#$ are \textit{connected by a split path $\mod Q$ of length $k$ in $\mathcal J^\#$} if there are primes $p_1,\dots,p_k\in\mathcal P_1, q_1,\dots,q_k\in\mathcal P_2$ and $(x_2,\vec\alpha_2),\dots,(x_k,\vec\alpha_k)\in\mathcal J^\#$ such that 
\[
|\frac{x_i}{p_i}-\frac{x_{i+1}}{q_i}|\ll\frac{H}{P} \quad\text{ and }\quad p_i^j\alpha^{(j)}_i \equiv q_i^j\alpha_{i+1}^{(j)}+O(\eta^{-O(1)}(P/H)^j)\mod Q\qquad (1\leq j\leq d).
\]
We use $l$ to denote a path and  the $3k+1$-tuple $(p_1,\dots,p_k,q_1,\dots,q_k,(x_1,\vec\alpha_1),\dots,(x_{k+1},\vec\alpha_{k+1}))$  indeed determine a path by the above two inequalities. We call $(x_1,\vec\alpha_1)$ the  \textit{starting point} and $(x_{k+1},\vec\alpha_{k+1})$ the \textit{endpoint} of this path. If $l_1$ and $l_2$ are two split paths $\mod Q$ of length $k$ connecting $(x,\vec\alpha)$ and $(y,\vec\beta)$ and sharing no prime in common, we say they are \textit{disjoint}.
\end{definition}

\begin{definition}[Regularity]\label{regular}
Let $Q\in\N$. We say $\mathcal J^\#\subseteq \mathcal J$ is 
\textit{$(c,Q)$-regular} if $|\mathcal J^\#|\geq c\eta^{O(1)}|\mathcal J|$ and for every $(x,\vec\alpha)\in\mathcal J^\#$ we can find $\gg\eta^{O(1)}(\frac{P}{\log P})^2$ of $(p,q,(y,\vec\beta)) \in \mathcal P_1 \times \mathcal P_2 \times \mathcal J^\#$ such that $(p,q,(x,\vec\alpha),(y,\vec\beta))$ forms a split path $\mod Q$ of length 1 in $\mathcal J^\#$.
	
\end{definition}

In view of Proposition \ref{connected},  for $\gg\eta^{O(1)}$ proportion of primes $p'\in[P',2P']$, we can find a $(c,p')$-regular subset of $\mathcal J$, and this is indeed \cite[Lemma 3.9]{Wal23-1}.

\begin{lemma}\label{regular-p}
Suppose that	 $P'$ and $\mathcal J$ are as in Proposition \ref{connected}, then for $\gg\eta^{O(1)}P'/\log P'$  primes $p'\in[P',2P']$ we can find a $(c,p')$-regular subset $\mathcal J_{p'}$ of $\mathcal J$.
\end{lemma}

The next result is the main conclusion of this section, which denominates that there is some $(x,\vec\alpha)\in \mathcal J$ and a large subset $\mathcal J^\#$ of $\mathcal J$ such that for each $(y,\vec\beta)\in\mathcal J^\#$ there are two disjoint reasonable long split paths connecting $(x,\vec\alpha)$ and $(y,\vec\beta)$. The proof essentially follows from  \cite[Lemma 5.4]{Wal23-1}, since we show the dependence of $\eta$ here  (which was not obvious in \cite{Wal23-1}) and the major restriction on $H$ will be revealed, we'd like to provide a detailed proof.

\begin{proposition}[Many pairs of disjoint split paths produced]\label{disjoint-paths}
Let $0<\sigma<1$ be a fixed number and $k_0=\floor{\frac{\log(X/(AH\log X))}{2\log(2P)}}$ for some sufficiently large constant $1\leq A\ll1$. Suppose that $0<\omega<1$ is defined in Section 4 that is sufficiently small with respect to $\sigma$, and $\exp\bigbrac{C(\log X)^{1/2}(\log\log X)^{1/2}}\leq H\leq X^{1/3-\eps}$ with $C>0$ sufficiently large with respect to $\sigma$ and $\omega$. Then there is an element $(x_0,\vec\alpha_0)\in \mathcal J$, a subset $\mathcal J^\#\subseteq \mathcal J$ with cardinality $\#\mathcal J^\#\gg X/H^{1+\sigma}$ such that the following holds.

For every $(y,\vec\beta)\in\mathcal J^\#$ there are two disjoint split paths $\mod{Q_y}$ of length $k_0+2$ connecting $(x,\vec\alpha)$ and $(y,\vec\beta)$, where $Q_y$  is  a product of $\gg\eta^{O(k_0)}\frac{P'}{\log P'}$ of primes $p'\in[P',2P']$.
	
\end{proposition}

\begin{proof}
Recording what we mean by the regularity in Definition \ref{regular}, Lemma \ref{regular-p} shows that
\[
\eta^{O(1)}\frac{P'}{\log P'}\frac{X}{H}
\ll\sum_{p'\sim P'} \#\mathcal J_{p'}
\ll\sum_{(x,\vec\alpha)\in\mathcal J}\sum_{p'\sim P'}1_{(x,\vec\alpha)\in\mathcal J_{p'}},
\]
one can see, from the pigeonhole principle, that there is some $(x,\vec\alpha)\in\mathcal J$ that is also in $\gg\eta^{O(1)}\frac{P'}{\log P'}$ of regular $\mathcal J_{p'}$. This $(x,\vec\alpha)$ is the $(x_0,\vec\alpha_0)$ described in the statement.

Fix a prime $p'\in[P',2P']$ for which $\mathcal J_{p'}$ is regular and  fix the above $(x,\vec\alpha)\in\mathcal J_{p'}$. By the regularity of $\mathcal J_{p'}$ (see Definition \ref{regular}),  there would be $\gg\eta^{O(k_0)}(\frac{P}{\log P})^{2k_0+4}$ split paths $\mod{p'}$ of length $k_0+2$ with  starting point $(x,\vec\alpha)$ in $\mathcal J_{p'}$. As for any given $2k_0+4$-tuple ordered primes $(p_1,\dots,q_{k_0+2})$, starting with $(x,\vec\alpha)$, there is an unique  split path $\mod 1$ of length $k_0+2$, hence, there are in total $O((\frac{P}{\log P})^{2k_0+4})$ split path $\mod 1$ of length $k_0+2$ with starting point $(x,\vec\alpha)$ in $\mathcal J$. Consequently, we can suppose that there are $\gg\eta^{O(k_0)}(\frac{P}{\log P})^{2k_0+4}$ split path $\mod 1$ of length $k_0+2$ with starting point $(x,\vec\alpha)$ in $\mathcal J$ that are also split paths $\mod{p'}$ for $\gg\eta^{O(k_0)}\frac{P'}{\log P'}$  primes $p'\in[P',2P']$. Write $\mathcal R$ for this set of paths.

Let $\mathcal R_y=\set{l\in\mathcal R: l\text{ has endpoint }(y,\vec\beta)}$. On recalling that for each $l\in\mathcal R_y$ there are $\gg\eta^{O(k_0)}\frac{P'}{\log P'}$  primes $p'\in[P',2P']$ such that $l\mod{p'}$ is also a path in $\mathcal J_{p'}$, we may use Cauchy-Schwarz inequality to produce pair of paths with the same endpoint $(y,\vec\beta)$ as follows.
\[
\eta^{O(k_0)} (\frac{P'}{\log P'})^2 |\mathcal R_y|^2\ll\Bigbrac{\sum_{l\in\mathcal R_y}\sum_{p'\sim P'}1_{l\mod{p'}}}^2\ll\frac{P'}{\log P'}\sum_{p'\sim P'}\sum_{l_1,l_2\in\mathcal R_y} 1_{l_1\mod{p'}} 1_{l_2\mod{p'}},
\]
by mean of pigeonhole principle there are $\gg \eta^{O(k_0)} |\mathcal R_y|^2$ of pairs of split paths $(l_1,l_2)\in \mathcal R_y^2$ that are also split paths $\mod{p'}$ for $\gg\eta^{O(k_0)}\frac{P'}{\log P'}$  primes $p'\in[P',2P']$. We can further count the number of pairs of paths with the same endpoint in $\mathcal R$. In practice, from the assumption,
\begin{multline*}
\eta^{O(k_0)} (\frac{P'}{\log P'})^2 (\frac{P}{\log P})^{4k_0+8}\ll\Bigbrac{\sum_{(y,\vec\beta)\in\mathcal J}\sum_{l\in\mathcal R_y}\sum_{p'\sim P'}1_{l\mod{p'}}}^2\\
\ll\frac{X}{H}	\frac{P'}{\log P'}\sum_{p'\sim P'}\sum_{(y,\vec\beta)\in\mathcal J} \sum_{l_1,l_2\in\mathcal R_y}
1_{l_1\mod{p'}} 1_{l_2\mod{p'}},
\end{multline*}
which means that the number of pairs of split paths in $\mathcal R$ with the same endpoint and also being paths $\mod{p'}$ for $\gg\eta^{O(k_0)}\frac{P'}{\log P'}$  primes $p'\in[P',2P']$ is at least
\[
\gg\eta^{O(k_0)}(\frac{P}{\log P})^{4k_0+8}\frac{H}{X}.
\]
Write this set of pairs of paths as $\mathcal R_0$, then $\mathcal R_0\subseteq\mathcal R\times\mathcal R$ and $\#\mathcal R_0\gg\eta^{O(k_0)}(\frac{P}{\log P})^{4k_0+8}\frac{H}{X}$.

Next, we need to rule out those pairs $(l_1,l_2)\in\mathcal R_0$ that are not disjoint (i.e. sharing a prime in common from Definition \ref{path}). Thanks to \cite[Corollary 5.3]{Wal23-1} which shows that there are not many pairs of paths of reasonable length with the same starting point and endpoint that can share a prime in common, we can conclude by applying \cite[Corollary 5.3]{Wal23-1} with $k=k_0+2$ that
\begin{multline*}
\#\set{(l_1,l_2)\in\mathcal R_0: l_1\text{ and }l_2\text{ are not disjoint}}\\
\ll(\frac{P}{\log P})^{4k_0+8}\frac{H}{X}\frac{\log^{4k_0+7}P}{P}k_02^{4k_0+7}(2k_0+4)!^2
\ll\eta^{O(k_0)}(\frac{P}{\log P})^{4k_0+8}\frac{H}{X},
\end{multline*}
by our hypothesis on $\omega$ and $C$. Hence, there is $\mathcal R^*\subseteq \mathcal R_0$ with $|\mathcal R^*|\gg \eta^{O(k_0)}(\frac{P}{\log P})^{4k_0+8}\frac{H}{X}$
and for each pair $(l_1,l_2)\in\mathcal R^*$ they are disjoint paths $\mod{p'}$ for $\gg\eta^{k_0}\frac{P'}{\log P'}$  primes $p'\in[P',P']$ (so that they are paths $\mod Q$ with $Q$ as the product of $\gg\eta^{k_0}\frac{P'}{\log P'}$ of primes $p'\in[P',P']$), besides, $l_1$ and $l_2$ are of length $k_0+2$ and having starting point $(x,\vec\alpha)$ and having the same endpoint. It is thus sufficient to count how many endpoints of the paths within $\mathcal R^*$. Write $\mathcal J^\#$ as the set of endpoints of paths in $\mathcal R^*$.

Given an element $(y,\vec\beta)\in\mathcal J$ we now consider  how many split paths of length $k_0+2$ with starting point $(x,\vec\alpha)$ and endpoint $(y,\vec\beta)$ can be constructed in $\mathcal J$. Let $l$ be one of such paths and decompose it as $l=l_1+l_2$, where $l_1$ is the path of length $k_0$ with starting point $(x,\vec\alpha)$ and $l_2$ is the path of length $2$ with endpoint $(y,\vec\beta)$. As there are only $O((\frac{P}{\log P})^4)$ such $l_2$ and each such $l_2$ would also give $l_1$ a fixed endpoint. It follows from \cite[Lemma 5.1]{Wal23-1} that given the starting point and endpoint, there are only $O((2k_0)!)$ split paths of length $k_0$. Therefore, there are at most $O((2k_0)!(\frac{P}{\log P})^4)$ split paths of length $k_0+2$ with starting point $(x,\vec\alpha)$ and endpoint $(y,\vec\beta)$. As a consequence,
\[
|\mathcal J^\#|\gg|\mathcal R^*|((2k_0)!)^{-2}(\frac{\log P}{ P})^8\gg \frac{X}{H^{1+\sigma}}
\]
if $H\geq\exp\bigbrac{C(\log X)^{1/2}(\log\log X)^{1/2}}$.
	
\end{proof}


\section{Building a global frequency and proving the main theorem}

Let $(x_0,\vec\alpha_0)\in\mathcal J$ be the frequency in Proposition \ref{disjoint-paths}, for every $(y,\vec\beta)\in \mathcal J^\#$ we can build a split path of length $2k_0+4$ which goes through $(y,\vec\beta)$ and has $(x_0,\vec\alpha_0)$ as both its starting and ending point. By adapting the argument in Subsection 2.2, we can obtain some information on the frequencies $\vec\alpha_0$ and $\vec\beta$.

\begin{lemma}\label{major-frequence}
Suppose that $0<\sigma<1$, $(x_0,\vec\alpha_0)\in\mathcal J$ and $(y,\vec\beta)\in \mathcal J^\#$ are as in Proposition 	\ref{disjoint-paths}, and $Q_y$ is the product of $\gg\eta^{O(k_0)}\frac{P'}{\log P'}$ primes in $[P',2P']$. Then for every degree $1\leq j\leq d$ there is always some $T_{j,y}\in\R$ with $|T_{j,y}|\ll\frac{X^{j}}{H^{j-\sigma}}$ and some constant $0<c_{j,y}\ll1$ such that 
\[
\alpha_0^{(j)}\equiv \frac{a^{(j)}_y}{q^{(j)}_y}\cdot Q_y+\frac{T_{j,y}}{y^j}+O\bigbrac{H^{-j+\sigma}}\mod {Q_y}
\]
and
\[
\beta^{(j)}\equiv \frac{b^{(j)}_y}{q^{(j)}_y}\cdot Q_y+\frac{c_{j,y}T_{j,y}}{y^j}+O\bigbrac{H^{-j+\sigma}}\mod {Q_y},
\]
where $a^{(j)}_y,b^{(j)}_y,q^{(j)}_y\ll H^{\sigma/d}$ are integers.\end{lemma}

\begin{proof}

In the proof we let $k=k_0+2$, where $k_0$ is the integer defined in Proposition \ref{disjoint-paths}. As mentioned earlier, given an element $(y,\vec\beta)\in\mathcal J_0$, one can construct, from Proposition \ref{disjoint-paths},  a split path modulo $Q_y$ of length $2k$ with both starting and ending points being $(x_0,\vec\alpha_0)$. Suppose that this path is associated with the prime tuple $(p_1,\dots,p_k,q_k',\dots,q_1';q_1,\dots,q_k,p_k',\dots,p_1')$, the path is as the  figure below.
\[\xymatrix@=8ex{
 {\vec \alpha_0 =\vec\alpha_{1,1}}  \ar@{-}[r]^{p_1}_{q_1} 
& \vec\alpha_{1,2} \ar@{-}[r]^{p_2}_{q_2} 
 & \cdots \ar@{-}[r]^{p_k}_{q_k}  
 &{\vec\alpha_{1,k+1}=\vec\beta}\ar@{-}[r]^{{q'_k}}_{{p'_k}}
 &\cdots \ar@{-}[r]^{{q'_1}}_{{p'_1}}
 &{\vec\alpha_{1,2k+1}=\vec \alpha_0}
 }
\]

Since $Q_y$ is the product of primes in the interval $[P',2P']$, it is evident from the choices of $P$ and $P'$ in Proposition \ref{connected} that $(p,Q_y)=1$ whenever $p\in[P,2P]$. Moreover, \cite[Lemma 4.1]{Wal23-1} demonstrates that $\prod_{i\leq m_1}\frac{p_i}{q_i}\prod_{i\leq m_2}\frac{p_i'}{q_i'}\asymp 1$ for $1\leq m_1,m_2\leq k$. Therefore, the conditions of Lemma \ref{top-element} are satisfied. 

Setting $\eps=\eta^{-O(1)}\frac{P}{H}$ in Lemma \ref{top-element} and assuming that $\vec\alpha_y$ is the top element, the application of Lemma \ref{top-element} with $i'=0$ and $i'=2k$, along with the triangle inequality, shows that	\[
	\prod_{i\leq k}(p_iq'_i)^j\alpha_y^{(j)} \equiv \prod_{i\leq k}(p'_iq_i)^j\alpha_y^{(j)} +O(k \eta^{-O(1)} (P/H)^j) \mod {Q_y}.
	\]
	Now we set $T_{j,y}\in\R $ as the number that satisfies
	\[
	\Bigbrac{\prod_{i\leq k}(p_iq_i')^j-\prod_{i\leq k}(p_i'q_i)^j} \alpha_y^{(j)}\equiv \frac{\Bigbrac{\prod_{i\leq k}(p_iq_i')^j-\prod_{i\leq k}(p_i'q_i)^j} T_{j,y}}{y^j\prod_{i\leq k}(p_iq_i')^j} \mod{Q_y}
	\]
	and 
	\[
	\Bigabs{\frac{\Bigbrac{\prod_{i\leq k}(p_iq_i')^j-\prod_{i\leq k}(p_i'q_i)^j} T_{j,y}}{y^j\prod_{i\leq k}(p_iq_i')^j}}=\Bigabs{\frac{T_{j,y}}{y^j}\bigbrac{1-\prod_{i\leq k}(p_i/p_i')^j(q_i'/q_i)^j}}\ll k\eta^{-O(1)}(P/H)^j.
	\]
It follows from the assumption of Proposition \ref{disjoint-paths} that $p_i,p_i',q_i,q_i'(1\leq i\leq k)$ are distinct primes, and thus, $\prod_{i\leq k}p_iq_i'/(p_i'q_i)\neq 1$;  on the other hand, one can also deduce from \cite[Lemma 4.1]{Wal23-1} and the setting of our $k$ that $ \prod_{i\leq k}p_iq_i'/(p_i'q_i)\ll 1$. Therefore, one has
\[
|T_{j,y}|\ll X^jk\eta^{-O(1)}(P/H)^j \ll X^j/H^{j-\sigma}
\]
since $P\ll H^\omega\ll H^\frac{\sigma}{dk}$ in light of $\omega$ is sufficiently small concerning  $\sigma$. Dividing $\prod_{i\leq k}(p_iq_i')^j-\prod_{i\leq k}(p_i'q_i)^j$ both sides from the above congruence equation, one has
\begin{align}\label{alpha-top}
\alpha_y^{(j)} \equiv\frac{c_y^{(j)}}{q_y^{(j)}}\cdot Q_y+\frac{T_{j,y}}{y^j \prod_{i\leq k}(p_iq_i')^j} \mod{Q_y}
\end{align}
where $1\leq c_y^{(j)}\leq q_y^{(j)}$ are integers and $q_y^{(j)}|(\prod_{i\leq k}(p_iq'_i)^j-\prod_{i\leq k}(p'_iq_i)^j)\ll H^{\sigma/d}$ since $p_i,p_i',q_i,q'_i\ll H^\omega\ll H^{\sigma/2kd}$ from Proposition \ref{connected}. 

Meanwhile, Lemma \ref{top-element} also suggests that there is some integer $1\leq a^{(j)}_y\leq q^{(j)}_y$ such that
\[
\alpha_0^{(j)}\equiv \prod_{i\leq k}(p_iq'_i)^j\alpha_y^{(j)}+O(k\eta^{-O(1)}(P/H)^j)\equiv\frac{a_y^{(j)}}{q_y^{(j)}}\cdot Q_y+\frac{T_{j,y}}{y^j}+O(H^{-j+\sigma}) \mod {Q_y}.
\]
Moreover, on combining (\ref{alpha-top}) with Lemma \ref{top-element} one also has
\begin{multline*}
\beta^{(j)} \equiv \prod_{i\leq k}(q_iq_i')^j\alpha_y^{(j)}+O(k\eta^{-O(1)}(P/H)^j) \equiv \frac{b_y^{(j)}}{q^{(j)}_y}\cdot Q_y+\frac{T_{j,y}}{y^j}\frac{\prod_{i\leq k}(q_iq_i')^j}{\prod_{i\leq k}(p_iq_i')^j}+O(H^{-j+\sigma}) \\
\equiv\frac{b_y^{(j)}}{q_y^{(j)}}\cdot Q_y+\frac{c_{j,y}T_{j,y}}{y^j}+O(H^{-j+\sigma})\mod{Q_y},
\end{multline*}
for some constant $0<c_{j,y}\ll1$, as one has $\prod_{i\leq k}p_i\asymp\prod_{i\leq k}q_i$ by employing \cite[Lemma 4.1]{Wal23-1}. This completes the proof.

\end{proof}

\vspace{2mm}

\noindent\emph{Proof of Theorem \ref{main}}.

\vspace{2mm}
Since for $X^{1/3-\eps}\leq H\leq X$ the claim follows immediately from  \cite[Theorem 1.1]{MRSTT}, let's suppose that $\exp\bigbrac{C(\log X)^{1/2}(\log\log X)^{1/2}}\leq H\leq X^{1/3-\eps}$ in the following. Assume 
\[
\int_X^{2X}\sup_{\vec\alpha\in\T^d} \Bigabs{\sum_{x<n\leq x+H} (\tau_k(n)-\tau^*_k(n)) e\bigbrac{\sum_{1\leq j\leq d}\alpha^{(j)}(n-x)^j}}\,\rd x\geq\eta XH \log^{k-1}X,
\]
it follows from the second conclusion of Proposition \ref{connected} that there is a $(c\eta^{2},H)$-configuration $\mathcal J\subseteq[X,2X]\times \T$ such that when $(y,\vec\beta_y)\in\mathcal J$  the  correlation below holds:
\begin{align}\label{s}
\Bigabs{ \sum_{y<n\leq y+H} (1_S(n)\tau_k(n)-\tau_k^*(n))  e\bigbrac{\sum_{1\leq j\leq d}\beta_y^{(j)}(n-y)^j}   } \gg \eta H\log^{k-1}X.
\end{align}
	 Proposition \ref{disjoint-paths} then gives us a set $\mathcal J^\#\subseteq\mathcal J$ with $|\mathcal J^\#|\gg\frac{|\mathcal J|}{H^\sigma}$ such that for every $(y,\vec\beta_y)\in\mathcal J^\#$, there exists a number $Q_y$ that is the product of $\gg\eta^{O(k_0)}P'\log P'$ primes $p'\in[P',2P']$, and $(x_0,\vec\alpha_0)$ and $(y,\vec\beta_y)$ are connected by a split path modulo $Q_y$. Applying the pigeonhole principle,  there exists $\tilde Q\geq P'^{c\eta^{O(k_0)}P'/\log P'}\geq\exp\bigbrac{c\eta^{O(k_0)}P'}$ such that for $\gg\eta^{O(k_0)}|\mathcal J^\#|\gg H^{-\sigma}|\mathcal J^\#|$ elements $(y,\vec\beta_y)\in\mathcal J^\#$, each of them can be connected with $(x_0,\vec\alpha_0)$ by a split path modulo $\tilde Q$. This implies that there exists $(y_0,\vec\beta_0)\in\mathcal J^\#$ such that
\[
\tilde Q_y=\gcd(Q_{y_0},Q_y)\geq\tilde Q\geq\exp\bigbrac{c\eta^{O(k_0)}P'}
\]
holds for $\gg H^{-\sigma}|\mathcal J^\#|$ of $(y,\vec\beta_y)\in\mathcal J^\#$. Let $T_j=T_{j,y_0}$, the application of Lemma \ref{major-frequence} with $(y_0,\vec\beta_0)$ and $(y,\vec\beta_y)$, respectively, shows that
\[
\alpha_0^{(j)} \equiv \frac{a^{(j)}_y}{q^{(j)}_y}\cdot Q_y+\frac{T_{j,y}}{y^j}+O(H^{-j+\sigma})\equiv \frac{a^{(j)}_{y_0}}{q^{(j)}_{y_0}}\cdot Q_{y_0}+\frac{T_j}{y_0^j}+O(H^{-j+\sigma})\mod {\tilde{Q_y}},
\]
which means that $\frac{a^{(j)}_y}{q^{(j)}_y}\cdot Q_y \equiv\frac{a^{(j)}_{y_0}}{q^{(j)}_{y_0}}\cdot Q_{y_0}\mod {\tilde{Q_y}}$ and $\bigabs{\frac{T_{j,y}}{y^j}-\frac{T_j}{y_0^j}}\ll H^{-j+\sigma}$. Since both $y$ and $y_0$ are in the interval $[X/(2P),2X/P]$ we may set $y/c_y=y_0$ with $1/2<c_y<2$, then 
\[
\bigabs{\frac{T_{j,y}}{y^j}-\frac{c_y^jT_j}{y^j}}\ll H^{-j+\sigma}.
\]
 Therefore, one can deduce from the above inequality and Lemma \ref{major-frequence} that there are $\gg\frac{X}{H^{1+2\sigma}}$ of $(y,\vec\beta_y)\in\mathcal J^\#$ such that for each $y$ there is some constant $0<c_y\ll1$ such that
\[
\beta_y^{(j)}\equiv\frac{b^{(j)}_y}{q_y^{(j)}}\cdot Q_y+\frac{c_{j,y}T_{j,y}}{y^j}+O(H^{-j+\sigma}) \equiv \frac{b^{(j)}_y}{q_y^{(j)}}\cdot Q_y+\frac{c_{j,y}c^j_yT_j}{y^j}+O(H^{-j+\sigma})\mod{\tilde Q_y}.
\]
As there are only  finite number of such $c_{j,y}c^j_y$, applying the pigeonhole principle allows us to assume that there is one such $c>0$ and $\gg\frac{X}{H^{1+2\sigma}}$ elements $(y,\vec\beta_y)\in\mathcal J$ such that
\[
\beta_y^{(j)}\equiv \frac{b^{(j)}_y}{q_y^{(j)}}\cdot Q_y+\frac{cT_j}{y^j}+O(H^{-j+\sigma})\mod{\tilde Q_y}
\]
for some $1\leq b_y^{(j)}\leq q_y^{(j)}$. We then rename $cT_j$ as $T_j$. As it follows from Lemma \ref{major-frequence} that for every $(y,\vec\beta_y)\in\mathcal J^\#$ we have $q_y^{(d)},\dots,q_y^{(1)}\ll H^{\sigma/d}$, the pigeonhole principle allows us to find a $k$-tuple $q^{(d)},\dots,q^{(1)}\ll H^{\sigma/d}$ such that $q_y^{(j)}=q^{(j)}$ for $\gg\frac{X}{H^{1+3\sigma}}$ of such $(y,\vec\beta_y)$. \cite[Proposition 3.13]{MRTTZ} further lowers these popular denominators $q^{(j)}$ to $O(1)$. Let's take $q$ as the product of $q^{(1)},\dots,q^{(d)}$, then we still have $q\ll 1$. Therefore, one may conclude from the above analysis along with inequality (\ref{s})  that there are integers $1\leq b^{(j)}\leq q$ such that for $\gg\eta^{O(1)}\frac{X}{H^{1+3\sigma}}$ of elements $(y,\vec\beta_y)\in[X,2X]\times\T^d$ satisfying $y$ are $H$-separated and
\[
\Bigabs{\sum_{y<n\leq y+H} (1_S(n)\tau_k(n)-\tau_k^*(n)) e\bigbrac{\sum_{1\leq j\leq d}(n-y)^j\beta_y^{(j)}}}
 \gg \eta H\log^{k-1}X,	
\]
where $\beta_y^{(j)}\equiv \frac{b^{(j)}}{q}\cdot Q_y+\frac{T_j}{y^j}+O(H^{-j+\sigma})\mod{\tilde Q_y}$.

By decomposing the interval $[y,y+H]$ into disjoint subintervals of length $\Omega(H^{1-3\sigma})$,  the pigeonhole principle gives us an integer $0\leq z_y-y\leq H$ such that for $\gg\eta^{O(1)}\frac{X}{H^{1+3\sigma}}$ of elements $(y,\vec\beta_y)\in[X,2X]\times\T^d$ the following inequality holds:
\[
\eta H^{1-3\sigma}\log^{k-1}X\ll \Bigabs{\sum_{z_y<n\leq z_y+H^{1-3\sigma}} (1_S(n)\tau_k(n)-\tau_k^*(n)) e\Bigbrac{\sum_{1\leq j\leq d}(n-y)^j\beta_y^{(j)}}};
\]
besides, by noting that $|T_j/z_y^j-T_j/y^j|\ll H^{-j+\sigma}$ we can also rewrite $\beta_y^{(j)}$ as
\[
\beta_y^{(j)}\equiv \frac{b^{(j)}}{q}\cdot Q_y+\frac{T_j}{z_y^j}+O(H^{-j+\sigma})\mod{\tilde Q_y}.
\]

If we fix a degree $1\leq j\leq d$ for a moment and consider the phase function $e\bigbrac{\set{(n-y)^j-(z_y-y)^j}\beta_y^{(j)}}$. By  Taylor expansion and noting that $|n-z_y|\leq H^{1-3\sigma}$ and $|z_y-y|\leq H$, this phase function equals to
\begin{multline*}
	e\Bigbrac{\set{(n-y)^j-(z_y-y)^j}\bigbrac{\frac{b_{j}}{q}+\frac{T_j}{z_y^j}}}e\Bigbrac{O\bigbrac{H^{-j+\sigma}\sum_{1\leq i\leq j}(n-z_y)^i(z_y-y)^{j-i}}}\\
	=e\Bigbrac{\set{(n-y)^j-(z_y-y)^j}\bigbrac{\frac{b_{j}}{q}+\frac{T_j}{z_y^j}}}\bigbrac{1+O(H^{-\sigma})},
\end{multline*}
where $b_j$ is an integer satisfying $b_j\equiv b^{(j)}Q_y\mod q$.
Therefore,
\begin{multline*}
\eta H^{1-3\sigma}\log^{k-1}X\ll\Bigabs{\sum_{z_y<n\leq z_y+H^{1-3\sigma}} (1_S(n)\tau_k(n)-\tau_k^*(n)) \prod_{1\leq j\leq d}e\Bigbrac{\frac{b_j(n-y)^j}{q}+\frac{T_j}{z_y^j}(n-y)^j}}	\\
+H^{-\sigma}\sum_{z_y<n\leq z_y+H^{1-3\sigma}}(\tau_k(n)+\tau_k^*(n))\\
\ll\sum_{a\mod {q}} \biggabs{\twosum{z_y<n\leq z_y+H^{1-3\sigma}}{  n\equiv a\mod {q}} (1_S(n)\tau_k(n)-\tau_k^*(n)) e\Bigbrac{\sum_{1\leq j\leq d}\frac{T_j}{z_y^j}(n-y)^j}}	\\
+H^{-\sigma}\sum_{z_y<n\leq z_y+H^{1-3\sigma}}(\tau_k(n)+\tau_k^*(n))
\end{multline*}
for $\gg\eta^{O(1)}\frac{X}{H^{1+3\sigma}}$ of $H$-separated $z_y\in[X,2X]$. If we claim that there are at most $O(\frac{X}{H^{1+3.1\sigma}})$ of $H$-separated $z_y\in[X,2X]$ such that
\begin{align}\label{222}
\sum_{z_y<n\leq z_y+H^{1-3\sigma}}(\tau_k(n)+\tau_k^*(n))\gg H^{1-3\sigma}\log^{k-1}X,
\end{align}
recalling that $q\ll1$, it follows from the pigeonhole principle that there are at least $\gg\eta^{O(1)}\frac{X}{H^{1+3\sigma}}$ of $H$-separated $z_y\in[X,2X]$ and an integer $1\leq a\leq q$ such that
\[
\Bigabs{\sum_{z_y<n\leq z_y+H^{1-3\sigma} \atop n\equiv a\mod{q}} (1_S(n)\tau_k(n)-\tau_k^*(n)) e\Bigbrac{\sum_{1\leq j\leq d}\frac{T_j}{z_y^j}(n-y)^j}} \gg \eta H^{1-3\sigma}\log^{k-1}X.
\]
In light of the claim (\ref{222}) above, along with the fact that $|z_y-y|\ll H$ and $|z_y-n|\ll H^{1-3\sigma}$ whenever $n\in[z_y,z_y+H^{1-3\sigma}]$, we may further assume that there are at least $\gg\eta^{O(1)}\frac{X}{H^{1+3\sigma}}$ of $H$-separated $z_y\in[X,2X]$ and an integer $1\leq a\leq q$ such that
\[
\Bigabs{\sum_{z_y<n\leq z_y+H^{1-3\sigma} \atop n\equiv a\mod{q}} (1_S(n)\tau_k(n)-\tau_k^*(n)) e\Bigbrac{\sum_{1\leq j\leq d}\frac{T_j}{z_y^j}(n-z_y)^j}} \gg \eta H^{1-3\sigma}\log^{k-1}X.
\]

By representing the $d$-th polynomial $\frac{T_d}{z_y^d}(n-z_y)^d$ in the form 
\begin{multline*}
\frac{(-1)^{d-1}\bigbrac{(-1)^{d-1}dT_d}}{dz_y^d}(n-z_y)^d+ (1-1)\sum_{1\leq i\leq d-1} \frac{(-1)^{i-1}\bigbrac{(-1)^{d-1}dT_d}}{iz_y^i}(n-z_y)^i	\\
=(-1)^{d-1}dT_d(\log n-\log {z_y}) - \sum_{1\leq i\leq d-1} \frac{(-1)^{i-1}\bigbrac{(-1)^{d-1}dT_d}}{iz_y^i}(n-z_y)^i +O(H^{-\sigma}),
\end{multline*}
the second inequality follows from Taylor expansion and the assumption $|T_d|\ll X^{d}/H^{d-\sigma}$ when $|n-z_y|\leq H^{1-3\sigma}$.  Inspired by this expansion, one can deduce that
\begin{multline*}
\sum_{1\leq j\leq d}\frac{T_j}{z_y^j}(n-z_y)^j=\sum_{1\leq j\leq d}\sum_{1\leq i\leq j}\frac{(-1)^{i-1}(j(-1)^{j-1}T_j-(j+1)(-1)^jT_{j+1})}{i z_y^i}(n-z_y)^i\\
=\sum_{1\leq j\leq d}	\bigbrac{j(-1)^{j-1}T_j-(j+1)(-1)^jT_{j+1}}(\log n-\log {z_y})+O(H^{-\sigma}),
\end{multline*}
with the convenience of setting $T_{d+1}=0$ and the second equation follows from  Taylor expansion and the assumption $|T_j+T_{j+1}|\ll X^{j+1}/H^{j+1-\sigma}$. Therefore, considering the claim (\ref{222}), we may assume there is some $T^*$ with $|T^*|\leq X^{d+1}$ and  $\gg\eta^{O(1)}\frac{X}{H^{1+3\sigma}}$ of $H$-separated $y\in[X,2X]$ and an integer $1\leq a\leq q$ such that
\[
\Bigabs{\sum_{z_y<n\leq z_y+H^{1-3\sigma}  \atop n\equiv a\mod{q}} (1_S(n)\tau_k(n)-\tau_k^*(n)) n^{iT^*}}\gg \eta H^{1-3\sigma}\log^{k-1}X.
\]

Taking $h=H^{1-3\sigma}$ in the following, we now fix a value $z_y\in(X,2X]$ such that the above inequality holds. Suppose that $I\subseteq[z_y-h^{1-\sigma},z_y+h+h^{1-\sigma}]$ is an interval satisfying $|I|=h$ and $|I\bigtriangleup(z_y,z_y+h]|\ll \eta^2h$. It follows from the triangle inequality and Lemma \ref{discretization} (3) with $h_2=\eta^2h$, $z_{1,x}^{(2)}=z_y$ and $z_{2,x}^{(2)}=z_y+h$ that
\begin{multline*}
	\Bigabs{\twosum{n\in I}{n\equiv a\mod{q}} (1_S(n)\tau_k(n)-\tau_k^*(n)) n^{iT^*} }
	\geq \bigabs{\sum_{z_y<n\leq z_y+h \atop n\equiv a\mod{q}}(1_S(n)\tau_k(n)-\tau_k^*(n)) n^{iT^*}}\\
	 - O\Bigbrac{\sum_{|n-z_y|\leq \eta^{2}h}\tau_k(n)}-O\Bigbrac{\sum_{|n-z_y-h|\leq \eta^{2}h}\tau_k(n)}\gg\eta h\log^{k-1}X.
\end{multline*}

 Due to for each $z_y$, there would be at least $\gg \eta^2h$ such intervals $I$ produced, we can conclude that there are $\gg\eta^{O(1)}X/h^{6\sigma}$ integers $x\in(X,2X]$ such that
 \[
 \Bigabs{\sum_{x<n\leq x+h \atop n\equiv a\mod{q}}(1_S(n)\tau_k(n)-\tau_k^*(n)) n^{iT^*}}\gg \eta h\log^{k-1}X,
 \] 
 and this contradicts Proposition \ref{major-arc-approx}. 

Now, if we assume that the claim (\ref{222}) fails and translate the interval $[z_y,z_y+H^{1-3\sigma}]$ by any shift of size at most $O(H^{1-3.1\sigma})$ as done above, then we can conclude that there are $\gg\frac{X}{H^{\sigma}}$ integers $x\in[X,2X]$ such that
\[
\sum_{x<n\leq x+H^{1-3\sigma}}(\tau_k(n)+\tau_k^*(n))\gg H^{1-3\sigma}\log^{k-1}X.
\]
However, this contradicts with Lemma \ref{shiu-bound} and Propositions \ref{long-short-divisor}-\ref{long-short-approx} by noting that, supposing $h_1=X^{1-\frac{1}{100k}}$, for all but at most  $O(\frac{X}{H^{\sigma}})$ integers $x\in[X,2X]$ we have
\[
(H^{1-3\sigma})^{-1}\sum_{x<n\leq x+H^{1-3\sigma}}(\tau_k(n)+\tau_k^*(n))\ll h_1^{-1}\sum_{x<n\leq x+h_1}\tau_k(n)\ll\log^{k-1}X.
\]
This completes the proof of Theorem \ref{main}.
\qed


\bibliographystyle{plain} 

\renewcommand{\bibname}{} 

\bibliography{averaged_fourier.bib}

\end{document}